\documentclass[12pt]{amsart}
\usepackage[utf8]{inputenc}
\usepackage{amssymb}
\usepackage{mathrsfs}
\usepackage{graphicx}

\usepackage{partmac}

\theoremstyle{plain} 
\newtheorem{thm}{Theorem}[section]
\newtheorem{prop}[thm]{Proposition}
\newtheorem{lem}[thm]{Lemma}
\newtheorem{cor}[thm]{Corollary}
\theoremstyle{definition}
\newtheorem{defn}[thm]{Definition}
\newtheorem{rem}[thm]{Remark}
\newtheorem{ex}[thm]{Example}

\theoremstyle{plain}
\newtheorem{innercustomthm}{Theorem}
\newenvironment{customthm}[1]
  {\renewcommand\theinnercustomthm{#1}\innercustomthm}
  {\endinnercustomthm}

\renewcommand{\theta}{\vartheta}
\renewcommand{\phi}{\varphi}
\renewcommand{\epsilon}{\varepsilon}
\renewcommand{\subset}{\subseteq}
\renewcommand{\supset}{\supseteq}

\newcommand{\eqrot}{\rotatebox{90}{$=$}}
\newcommand{\subsetrot}{\rotatebox{90}{$\subset$}}
\newcommand{\supsetrot}{\rotatebox{90}{$\supset$}}
\newcommand{\subsetneqrot}{\rotatebox{90}{$\subsetneq$}}
\newcommand{\supsetneqrot}{\rotatebox{90}{$\supsetneq$}}

\newcommand{\N}{\mathbb N}
\newcommand{\Z}{\mathbb Z}

\newcommand{\C}{\mathbb C}
\DeclareMathOperator{\Mor}{Mor}
\DeclareMathOperator{\Rep}{Rep}
\DeclareMathOperator{\Reptil}{\widetilde{Rep}}
\DeclareMathOperator{\spanlin}{span}
\DeclareMathOperator{\Lrot}{Lrot}
\DeclareMathOperator{\Rrot}{Rrot}
\newcommand{\red}{_{\mathrm{red}}}
\newcommand{\nred}{_{N\mathchar `\-\mathrm{red}}}
\newcommand{\nlin}{_{N\mathchar `\-\mathrm{lin}}}
\newcommand{\nnlin}{_{(N+1)\mathchar `\-\mathrm{lin}}}
\newcommand{\nnnlin}{_{(N-1)\mathchar `\-\mathrm{lin}}}
\newcommand{\Lin}{\mathscr{L}}
\newcommand{\Cat}{\mathscr{C}}
\newcommand{\Kat}{\mathscr{K}}
\newcommand{\T}{\mathcal{T}}
\newcommand{\V}{\mathcal{V}}
\newcommand{\PP}{\mathcal{P}}
\newcommand{\Lart}{\mathscr{P}}

\newcommand{\tiltimes}{\mathbin{\tilde\times}}
\newcommand{\tilstar}{\mathbin{\tilde *}}

\begin{document}
\title{Intertwiner spaces of quantum group subrepresentations}
\author{Daniel Gromada}
\author{Moritz Weber}
\address{Saarland University, Fachbereich Mathematik, Postfach 151150,
66041 Saarbr\"ucken, Germany}
\email{gromada@math.uni-sb.de}
\email{weber@math.uni-sb.de}
\date{\today}
\thanks{Both authors were supported by the collaborative research centre SFB-TRR 195 ``Symbolic Tools in Mathematics and their Application''. The second author was also supported by the ERC Advanced Grant NCDFP, held by Roland Speicher and by the DFG project ``Quantenautomorphismen von Graphen''. The article is part of the first author's PhD thesis.}
\thanks{We thank Simeng Wang and Adam Skalski for discussions on the topic.}

\begin{abstract}
We consider compact matrix quantum groups whose $N$-dimensional fundamental representation decomposes into an $(N-1)$-dimensional and a one-dimensional subrepresentation. Even if we know that the compact matrix quantum group associated to this $(N-1)$-dimensional subrepresentation is isomorphic to the given $N$-dimensional one, it is a priori not clear how the intertwiner spaces transform under this isomorphism. In the context of so-called easy and non-easy quantum groups, we are able to define a transformation of linear combinations of partitions and we explicitly describe the transformation of intertwiner spaces. As a side effect, this enables us to produce many new examples of non-easy quantum groups being isomorphic to easy quantum groups as compact quantum groups but not as compact matrix quantum groups.
\end{abstract}

\maketitle
\section*{Introduction}

Compact (matrix) quantum groups were defined by Woronowicz in \cite{Wor87} by the following consideration. For a compact group $G$, we can construct a commutative C*-algebra $A:=C(G)$ of continuous complex-valued functions over $G$. The multiplication $\mu\colon G\times G\to G$ can be described by a \emph{comultiplication} $\Delta\colon A\to A\otimes A$. The group axioms can also be dualized and formulated in terms of the commutative algebra $A$ and the comultiplication $\Delta$. Such an alternative definition of a group can be generalized by dropping the commutativity condition on $A$. The resulting structure called \emph{compact quantum group} forms a counterpart of groups in non-commutative geometry. In particular, generalizing compact matrix groups, we get \emph{compact matrix quantum groups}, which are defined by some distinguished representation called \emph{fundamental representation}.

By a Tannaka--Krein result of Woronowicz \cite{Wor88}, compact matrix quantum groups are determined by their representation theory. The intertwiner spaces of a quantum group form a monoidal $*$-category. Conversely, any monoidal $*$-category of operators determines a quantum group if we interpret those operators as intertwiners.

In 2009, Banica and Speicher \cite{BS09} found an easy way how to construct such categories. They defined a structure of monoidal involutive category on set partitions and a functor assigning to each partition~$p$ a linear operator~$T_p$. Thus, for any subcategory~$\Cat$ of partitions, one can construct a so-called \emph{easy quantum group}~$G$ with fundamental representation~$u$, an $N\times N$ matrix for which the intertwiner spaces look like
$$\Mor(u^{\otimes k},u^{\otimes l})\equiv\{T\in\Lin(\C^{Nk},\C^{Nl})\mid Tu^{\otimes k}=u^{\otimes l}T\}=\spanlin\{T_p\mid p\in\Cat(k,l)\}.$$
It holds that the intertwiner spaces for the group of permutations~$S_N$ come from the category of all set partitions. On the other hand, the smallest category we usually consider is the category of non-crossing pair partitions, which corresponds to Wang's free orthogonal quantum group~$O_N^+$ defined in \cite{Wan95free}. Thus, for any easy quantum group~$G$ we have $S_N\subset G\subset O_N^+$.

This was a groundbreaking step in the theory of compact matrix quantum groups since it brought a lot of new examples of quantum groups and made some questions much easier to decide since we are now able to work on the combinatorial level rather than on the C*-algebraical level. On the other hand, this approach cannot describe all the quantum groups. In particular, it describes only those quantum groups $S_N\subset G\subset O_N^+$, whose intertwiner spaces are spanned by the maps~$T_p$ for some partitions~$p$.

Nevertheless, since for the group $S_N$ the intertwiner spaces are described by the set of all partitions
$$\Mor(u^{\otimes k},u^{\otimes l})=\spanlin\{T_p\mid \text{$p$ is a partition of $k+l$ points}\},$$
the intertwiner spaces of any quantum group $G$ such that $S_N\subset G\subset O_N^+$ must be some subspaces
$$\Mor(u^{\otimes k},u^{\otimes l})\subset\spanlin\{T_p\mid \text{$p$ is a partition of $k+l$ points}\}.$$
In order to be able to study all compact matrix quantum groups between $S_N$ and $O_N^+$, it is enough to introduce a linear structure to the category of partitions such that we are able to describe any subspace of $\spanlin\{T_p\}$, not only those having the maps $T_p$ as generators.

Despite the fact that the classical categories of partitions were already classified \cite{RW16}, we know almost nothing about the linear categories of partitions (categories of formal linear combinations of partitions). Working with linear combinations is much more complicated than working just with partitions. For example, given a linear combination of partitions, it is a quite non-trivial problem to determine whether the category generated by this linear combination is easy (i.e.\ spanned by partitions) or not. As far as we know, this article is one of the first papers containing amount of non-easy examples of linear categories of partitions.

A wide class of non-easy examples was recently constructed in \cite{Maa18}, where the intertwiner spaces for all so-called \emph{group-theoretical} quantum groups were described. There are also several results going beyond easiness using different approaches such as generalizing the notion of category of partitions or reinterpreting the partitions. See for example \cite{Fre17,CW16,Ban18super}.

Except for bringing examples of non-easy quantum groups, the main goal of this article is to study the following phenomenon. It is known that the bistochastic group $B_N$, which is the group of $N\times N$ matrices such that the sum of every row and the sum of every column equals to one, is isomorphic to $O_{N-1}$ \cite{BS09}. In fact, the fundamental representation of $B_N$ is reducible and decomposes into a direct sum of a one-dimensional trivial representation and an $(N-1)$-dimensional representation, which is similar to the fundamental representation of $O_{N-1}$. The same was proven for the free counterpart $B_N^+$ being isomorphic to $O_{N-1}^+$ \cite{Rau12}. It is also known that the representation of $S_N$ by permutation matrices decomposes into the one-dimensional trivial representation and $(N-1)$-dimensional standard representation. Thus, the free counterpart $S_N^+$ must also admit this decomposition. Our motivating question was: What is the $(N-1)$-dimensional compact matrix quantum group that is isomorphic to $S_N^+$? What do the intertwiner spaces look like?

In fact, every quantum group $G$ such that $S_N\subset G\subset B_N^{\# +}$ has a fundamental representation that decomposes into a one-dimensional and an $(N-1)$-dimensional subrepresentation. Recall that $B_N^{\# +}$ is the easy quantum group whose associated category of partitions is generated by the double singleton $\singleton\otimes\singleton$ \cite{Web13}. The crucial part is that we choose a very special isomorphism (in fact, we have two variants here) for passing from the $N$-dimensional representation of $G$ to the $(N-1)$-dimensional one allowing us to describe also the resulting intertwiner spaces using partitions or their linear combinations. In this article, we study the quantum group $H$ generated by this $(N-1)$-dimensional representation and we present a way, how to describe it using partitions. We also mention the opposite problem of how to reconstruct the whole quantum group $G$ from the group $H$. We summarize our results in the following section. Our considerations lead to many concrete examples of non-easy quantum groups, which are also summarized in the following section.

%

\section{Main results}

\subsection{Warm up example}
\label{secc.warmup}

Recall that the bistochastic quantum group $B_N^+$ is given by its fundamental representation $u=(u_{ij})_{i,j=1}^N$ and the relations turning $u$ into an orthogonal matrix, $u=\bar u$ together with the bistochastic relations
$$u\xi=\xi,\quad u^t\xi=\xi,$$
where $\xi\in\C^N$ is the vector filled with entries all equal to one. Now, for any orthogonal matrix $U\in M_N(\C)$ mapping $U\xi=\alpha e_N$ for some $\alpha\in\C$, we have
$$UuU^t=\begin{pmatrix}v&0\\0&1\end{pmatrix}=v\oplus 1$$
	with $v\in M_{N-1}(C(B_N^{+}))$. In addition, $v$ is orthogonal, so we have just proved the isomorphism $B_N^+\simeq O_{N-1}^+$, see \cite{Rau12}. A natural question is now: To which quantum subgroup of $O_{N-1}^+$ is $S_N^+$ isomorphic? How to describe its intertwiner spaces? Supprisingly, this question is much harder for $S_N^+$ than for $B_N^+$ and it heavily depends on the choice of the unitary $U$.

To make this more precise, let $G=(C(G),u)$ be any CMQG with $S_N\subset G\subset B_N^+$. Then, as the category of $B_N^+$ is generated by the singleton $\singleton$, we have again $u\xi=\xi$ and $u^t\xi=\xi$. Hence, for any orthogonal matrix $U$ as above, we infer that $G$ is isomorphic to some quantum group $S_{N-1}\subset G^{\rm irr}\subset O_{N-1}^+$. As this is an isomorphism as compact quantum groups (rather than as compact matrix quantum groups), we have no information about the intertwiner spaces of $G^{\rm irr}$ in general -- and they may be very different from the ones of $G$, as in the example $B_N^+\simeq O_{N-1}^+$.

We can make the statement even more general and consider a CMQG $S_N\subset G\subset B_N^{\# +}$, whose fundamental representation decomposes as $UuU^*=v\oplus r$, where $r\in C(G)$ is a one-dimensional representation of $G$. Then, the quantum group $G^{\rm irr}$ determined by the subrepresentation $v$ might not be isomorphic to $G$ itself, but essentially to some quotient $G/\hat\Z_2$.

\subsection{Main theorems}

We define particular orthogonal matrices $U_{(N,+)},U_{(N,-)}\in M_N(\C)$ such that we may describe the intertwiner spaces of $G^{\rm irr}$ explicitly for any CMQG $S_N\subset G\subset B_N^{\# +}$. To be more precise, we define transformations $\V_{(N,\pm)}\colon\Lart\nlin\to\Lart\nnnlin$ of linear combinations of partitions (Definition \ref{D.VK}) and two quantum groups $G_+^{\rm irr}$ and $G_-^{\rm irr}$ (Definition \ref{D.Girr}) using orthogonal matrices $U_{(N,+)}$ and $U_{(N,-)}$ (Definition \ref{D.U}) and we prove the following theorem.

\begin{customthm}{A}[Theorem~\ref{T.VG}, Theorem \ref{T.VK}]
\label{T1}
Let $G=(C(G),u)$, $S_N\subset G\subset B_N^{\# +}$ be a CMQG corresponding to a category $\Kat\subset\Lart\nlin$, $\singleton\otimes\singleton\in\Kat$. Then $G_\pm^{\rm irr}$ satisfies $S_{N-1}\subset G^{\rm irr}_\pm\subset O_{N-1}^+$ and it corresponds to the category
$$\V_{(N,\pm)}\Kat=\{\V_{(N,\pm)}p\mid p\in\Kat\}\subset\Lart\nnnlin.$$
\end{customthm}

Note that $\singleton\otimes\singleton\not\in\V_{(N,\pm)}\Kat$, so the fundamental representation of $G_\pm^{\rm irr}$ is no longer reducible. In the case $G=B_N^+$ (alternatively $B_N'^+$ or $B_N^{\# +}$), we have $G_+^{\rm irr}=G_-^{\rm irr}=O_{N-1}^+$.

Given a compact matrix quantum group $G$, using any regular matrix $T\in M_N(\C)$, we can construct a similar matrix quantum group $\tilde G=TGT^{-1}$. Although $G$ and $\tilde G$ are similar, they are not the same and in particular they have different (although similar) intertwiner spaces. In particular, even if $S_N\subset G\subset O_N^+$, this might not hold for $\tilde G$. If it does, easiness of $G$ need not imply easiness of $\tilde G$. If we are looking for examples of non-easy quantum groups, this provides the simplest construction for obtaining such examples. It turns out that there exists one canonical choice of a similarity matrix $T\in M_N(\C)$ such that for any $G$ satisfying $S_N\subset G\subset O_N^+$, we have also $S_N\subset\tilde G\subset O_N^+$. The following theorem describes this similarity relation explicitly and also gives an explicit description for the transformation of the corresponding linear category of partitions. This in particular provides the similarity between the quantum groups $G_+^{\rm irr}$ and $G_-^{\rm irr}$.

\begin{customthm}{B}[Theorem~\ref{T.Tiso} and Proposition~\ref{P.tauirr}]
\label{T3}
Let $G=(C(G),u)$, $S_N\subset G\subset O_N^+$ be a CMQG corresponding to a category $\Kat\subset\Lart\nlin$. Let us denote $\tau_{(N)}:=\idpart-\frac{2}{N}\disconnecterpart\in\Lart\nlin(1,1)$. Then $\tilde G:=(C(G),T_{\tau_{(N)}}uT_{\tau_{(N)}}^{-1})$ is a compact matrix quantum group satisfying $S_N\subset\tilde G\subset O_N^+$ and it corresponds to the category
$$\T_{(N)}\Kat=\{\tau_{(N)}^{\otimes l}p\tau_{(N)}^{\otimes k}\mid p\in\Kat(k,l)\}_{k,l\in\N_0},$$
where $\T_{(N)}\colon\Lart\nlin\to\Lart\nlin$ mapping $p\mapsto \tau_{(N)}^{\otimes l}p\tau_{(N)}^{\otimes l}$ for $p\in\Lart\nlin(k,l)$ provides a monoidal $*$-isomorphism $\Kat\to\T_{(N)}\Kat$.

In particular, considering $G$ such that $S_N\subset G\subset B_N^{\# +}$, we have the similarity $T_{\tau_{(N-1)}}G_+^{\rm irr}T_{\tau_{(N-1)}}^{-1}=G_-^{\rm irr}$, but on the other hand $T_{\tau_{(N)}}GT_{\tau_{(N)}}^{-1}=G$.
\end{customthm}

While Theorem~\ref{T1} described the construction of $G_\pm^{\rm irr}$ for a given $G$ such that $S_N\subset G\subset B_N^{\# +}$, the main result of Section \ref{sec.projection} shows that there are three canonical ways how to construct a quantum group $\tilde G$ such that $\tilde G^{\rm irr}_\pm=G^{\rm irr}_\pm$, and provides a description of the corresponding intertwiner spaces. In the following, we use the operator $\PP_{(N)}\colon\Lart\nlin(k,l)\to\Lart\nlin(k,l)$ mapping $p\mapsto\pi^{\otimes l}p\pi^{\otimes k}$, where $\pi=\idpart-\frac{1}{N}\disconnecterpart\in\Lart\nlin(1,1)$ (see Definition \ref{D.P}).

\begin{customthm}{C}[Theorem~\ref{T.redqg}]
\label{T2}
Let $\Kat\subset\Lart\nlin$ be a linear category of partitions such that $\singleton\otimes\singleton\in\Kat$. Denote by $G$ the corresponding quantum group. Then we can construct the quantum group corresponding to the following categories
\begin{eqnarray*}
\langle\PP_{(N)}\Kat\rangle\nlin           &\qquad\hbox{corresponds to}\qquad& U_{(N,\pm)}^*(G^{\rm irr}_\pm*\hat\Z_2)U_{(N,\pm)},\\
\langle\PP_{(N)}\Kat,\Labac\rangle\nlin    &\qquad\hbox{corresponds to}\qquad& U_{(N,\pm)}^*(G^{\rm irr}_\pm\times\hat\Z_2)U_{(N,\pm)},\\
\langle\PP_{(N)}\Kat,\singleton\rangle\nlin&\qquad\hbox{corresponds to}\qquad& U_{(N,\pm)}^*(G^{\rm irr}_\pm\times E)U_{(N,\pm)}.\\
\end{eqnarray*}
\end{customthm}

\subsection{Examples of non-easy categories}

In this subsection we summarize all examples of non-easy linear categories of partitions constructed in this work. Most of them come from Section \ref{sec.examples}. We define each category by a set of generators. The categories usually depend on a number $N\in\N$, which stands for the dimension of the matrix of the corresponding compact matrix quantum group. To simplify the formulation, we do not mention the precise compact matrix quantum group corresponding to those examples, but rather the well known quantum groups that are isomorphic to those. In each example, we give a reference to the precise statement.

Those examples were actually the true motivation for this work since big part of them was constructed by performing computer experiments using the computer algebraic software \textsc{Singular} \cite{Singular}. The theory presented in this article was developed afterwards to interpret those examples.

\begin{ex}[Proposition~\ref{P.Vqg}, Example \ref{ex.V}]
Applying Theorem~\ref{T1} on the quantum groups with categories
$$NC\nnlin=\langle\Laaa\rangle\nnlin\quad\text{and}\quad NC'\nnlin=\langle\Laaaa\rangle\nnlin$$
we get the following.
\begin{eqnarray*}
\langle\V_{(N+1,\pm)}\Laaa\rangle\nlin  &\qquad \longleftrightarrow \qquad& S_{N+1}^+,\\
\langle\V_{(N+1,\pm)}\Laaaa\rangle\nlin &\qquad\longleftrightarrow\qquad& S_{N+1}^+\times\hat\Z_2,\\
\end{eqnarray*}
where
$$\V_{(N+1,\pm)}\Laaa=\Laaa-{1\over N}\left(1\pm{1\over\sqrt{N+1}}\right)(\Laab+\Laba+\Labb)+{1\over N^2}\left(2\pm{N+2\over\sqrt{N+1}}\right)\Labc,$$
\begin{align*}
\V_{(N+1,\pm)}\Laaaa&=\Laaaa-{1\over N}\left(1\pm{1\over\sqrt{N+1}}\right)(\Laaab+\Laaba+\Labaa+\Labbb)+\\&+{1\over N^2}\left(\frac{N+2}{N+1}\pm{2\over\sqrt{N+1}}\right)(\Laabc+\Labac+\Labca+\\&+\Labbc+\Labcb+\Labcc)+\frac{1}{N^3}\left(\frac{N^2-4N-8}{N+1}\mp\frac{8}{\sqrt{N+1}}\right)\Labcd.
\end{align*}
\end{ex}

\begin{ex}[Proposition~\ref{p.qg}, Example \ref{ex.P}]
Applying Theorem~\ref{T2} on the previous examples we get the following
\begin{eqnarray*}
\langle\PP_{(N)}\Laaa,\singleton\otimes\singleton\rangle\nlin  &\qquad\longleftrightarrow\qquad& S_N^+*\hat\Z_2,\\
\langle\PP_{(N)}\Laaa,\Labac\rangle\nlin  &\qquad\longleftrightarrow\qquad& S_N^+\times\hat\Z_2,
\end{eqnarray*}
where
$$\PP_{(N)}\Laaa=\Laaa-{1\over N}(\Laab+\Laba+\Labb)+{2\over N^2}\Labc.$$
And
\begin{eqnarray*}
\langle \PP_{(N)}\Laaaa,\singleton\otimes\singleton\rangle\nlin  &\qquad\longleftrightarrow\qquad& (S_N^+\times\hat\Z_2)*\hat\Z_2,\\
\langle \PP_{(N)}\Laaaa,\Labac\rangle\nlin  &\qquad\longleftrightarrow\qquad& (S_N^+\times\hat\Z_2)\times\hat\Z_2,\\
\langle \PP_{(N)}\Laaaa,\singleton\rangle\nlin  &\qquad\longleftrightarrow\qquad& S_N^+\times\hat\Z_2,
\end{eqnarray*}
where
\begin{align*}
\PP_{(N)}\Laaaa&=\Laaaa-{1\over N}(\Laaab+\Laaba+\Labaa+\Labbb)+\\&+{1\over N^2}(\Laabc+\Labac+\Labca+\Labbc+\Labcb+\Labcc)+\frac{3}{N^3}\Labcd.
\end{align*}
\end{ex}

\begin{ex}[Proposition~\ref{P.cross}]
Applying Theorem~\ref{T2} on $\langle\crosspart,\singleton\otimes\singleton\rangle\nlin$, we get
\begin{eqnarray*}
\langle\PP_{(N)}\crosspart,\singleton\otimes\singleton\rangle\nlin  &\qquad\longleftrightarrow\qquad& O_{N-1}*\hat\Z_2,
\end{eqnarray*}
where
$$\PP_{(N)}\crosspart=\crosspart-{1\over N}(\Pabcb+\Pabac)+{1\over N^2}\Pabcd.$$
\end{ex}

\begin{ex}[Proposition~\ref{P.hl}]
Applying Theorem~\ref{T2} on $\langle\halflibpart,\singleton\otimes\singleton\rangle\nlin$, we get
\begin{eqnarray*}
\langle\PP_{(N)}\halflibpart,\singleton\otimes\singleton\rangle\nlin  &\qquad\longleftrightarrow\qquad& O_{N-1}^**\hat\Z_2,\\
\langle\PP_{(N)}\halflibpart,\Labac\rangle\nlin   &\qquad\longleftrightarrow\qquad& O_{N-1}^*\times\hat\Z_2,\\
\langle\PP_{(N)}\halflibpart,\singleton\rangle\nlin  &\qquad\longleftrightarrow\qquad& O_{N-1}^*,
\end{eqnarray*}
where
$$\PP_{(N)}\halflibpart=\halflibpart-{1\over N}(\Pabcdbc+\Pabcadc+\Pabcabd)+{1\over N^2}(\Pabcdec+\Pabcdbe+\Pabcade)-\frac{1}{N^3}\Pabcdef.$$
\end{ex}

\begin{ex}[Example \ref{E.tfour}]
Applying Theorem~\ref{T3} on $\langle\fourpart\rangle\nlin$, we get
\begin{eqnarray*}
\langle\T_{(N)}\fourpart\rangle\nlin  &\qquad\longleftrightarrow\qquad& H_N^+,
\end{eqnarray*}
where
\begin{eqnarray*}
\T_{(N)}\fourpart&=&\fourpart-{2\over N}(\Laaab+\Laaba+\Labaa+\Labbb)\\&+&{4\over N^2}(\Laabc+\Labac+\Labbc+\Labca+\Labcb+\Labcc)-{16\over N^3}\Labcd.
\end{eqnarray*}
\end{ex}

\section{Quantum groups and partitions}

In this preliminary section, we recall the basic notions of compact matrix quantum groups and Tannaka--Krein duality. For a more detailed introduction, we refer to the monographs \cite{Tim08,NT13}.

\subsection{Compact matrix quantum groups}
\label{secc.qgdef}
Let $A$ be a unital C*-algebra, $u_{ij}\in A$, where $i,j=1,\dots,N$ for some $N\in\N$. Denote $u:=(u_{ij})_{i,j=1}^N\in M_N(A)$. The pair $(A,u)$ is called a \emph{compact matrix quantum group} if
\begin{enumerate}
\item the elements $u_{ij}$, $i,j=1,\dots, N$ generate $A$,
\item the matrices $u$ and $u^t=(u_{ji})$ are invertible,
\item the map $\Delta\colon A\to A\otimes_{\rm min} A$ defined as $\Delta(u_{ij}):=\sum_{k=1}^N u_{ik}\otimes u_{kj}$ extends to a unital $*$-homomorphism.
\end{enumerate}

Compact matrix quantum groups are generalizations of compact matrix groups in the following sense. For $G\subseteq M_N(\C)$ we can take the algebra of continuous functions $A:=C(G)$. This algebra is generated by the functions $u_{ij}\in C(G)$ assigning to each matrix $g\in G$ its $(i,j)$-th element $g_{ij}$. The so-called \emph{co-multiplication} $\Delta\colon C(G)\to C(G)\otimes C(G)\simeq C(G\times G)$ is connected with the matrix multiplication on $G$ by $\Delta(f)(g,h)=f(gh)$ for $f\in C(G)$ and $g,h\in G$.

Therefore, for a general compact matrix quantum group $G=(A,u)$, the algebra $A$ should be seen as an algebra of non-commutative functions defined on some underlying non-commutative compact space. For this reason, we often denote $A=C(G)$ even if $A$ is not commutative. The matrix $u$ is called the \emph{fundamental representation} of $G$.

A compact matrix quantum group $H=(C(H),v)$ is a \emph{quantum subgroup} of $G=(C(G),u)$, denoted as $H\subseteq G$, if $u$ and $v$ have the same size and there is a surjective $*$-homomorphism $\phi\colon C(G)\to C(H)$ sending $u_{ij}\mapsto v_{ij}$. We say that $G$ and $H$ are \emph{identical} if there exists such a $*$-isomorphism (i.e.\ if $G\subset H$ and $H\subset G$). We say that $G$ and $H$ are \emph{similar} if there exists a regular matrix $T$ and a $*$-isomorphism $\phi\colon C(G)\to C(H)$ mapping $u_{ij}\mapsto [TvT^{-1}]_{ij}$. In this case, we write $G=THT^{-1}$.

One of the most important examples is the quantum generalization of the orthogonal group. The orthogonal group can be described as
$$O_N=\{U\in M_N(\C)\mid U_{ij}=\bar U_{ij},UU^t=U^tU=1\}.$$
So, it can be treated also as a compact matrix quantum group $(C(O_N),u)$, where $C(O_N)$ can be described as a universal C*-algebra
$$C(O_N)=C^*(u_{ij},\;i,j=1,\dots,N\mid u_{ij}=u^*_{ij},uu^t=u^tu=1,u_{ij}u_{kl}=u_{kl}u_{ij}).$$

This algebra can be quantized by dropping the commutativity relation. This was done by Wang in \cite{Wan95free} and the resulting quantum group described by the algebra
$$C(O_N^+):=C^*(u_{ij},\;i,j=1,\dots,N\mid u_{ij}=u^*_{ij},uu^t=u^tu=1)$$
is called the \emph{free orthogonal quantum group}.

\subsection{Representations of CMQG}

As in the case of classical groups, the defining or fundamental representation is not the only representation of a given group. For a compact matrix quantum group $G=(C(G),u)$, we say that $v\in M_n(C(G))$ is a representation of $G$ if $\Delta(v_{ij})=\sum_{k}v_{ik}\otimes v_{kj}$, where $\Delta$ is the comultiplication defined in the previous subsection. The representation $v$ is called \emph{unitary} if it is unitary as a matrix, i.e.\ $\sum_k v_{ik}v_{jk}^*=\sum_k v_{ki}^*v_{kj}=\delta_{ij}$.

Let $u\in M_n(C(G))$ and $v\in M_m(C(G))$. As in the case of ordinary matrices, we can consider the direct sum $u\oplus v\in M_{n+m}(C(G))$, the tensor product $u\otimes v\in M_{nm}(C(G))$ or the complex conjugate $\bar u=(u_{ij}^*)\in M_{n}(C(G))$. It is easy to check that if $u$ and $v$ are representations of some compact quantum group $G$, then those operations define new representations of $G$.

Let $u\in M_n(C(G))$ and $v\in M_m(C(G))$ be representations of a compact quantum group $G$. A linear map $T\colon\C^n\to\C^m$ is called an \emph{intertwiner} if $Tu=vT$. The space of all such maps is denoted $\Mor(u,v)$. The representations $u$ and $v$ are called \emph{equivalent} if there exists an invertible operator $T$ intertwining those representations.

We say that a subspace $V\subset\C^n$ is \emph{invariant} with respect to a representation $v\in M_n(C(G))$ if $vP=PvP$, where $P$ is the orthogonal projection onto $V$. It holds that any unitary representation $v$ is completely reducible. That is, if $V=P\C^n$ is an invariant subspace, then $V^\perp=(I-P)\C^n$ is also invariant. From the equality $v(I-P)=(I-P)v(I-P)$, we can see that $PvP=Pv$. So, $V$ is an invariant subspace of $v$ if and only if the corresponding projection $P$ intertwines $v$ with itself $vP=Pv$.

A representation $v\in M_n(C(G))$, whose only invariant subspaces are $\{0\}$ and $\C^n$ is called \emph{irreducible}. From the complete reducibility it follows that every unitary representation is a direct sum of irreducible ones.

\subsection{Monoidal $*$-categories}
Let $R$ be a set of \emph{objects}. For every $r,s\in R$, let $\Mor(r,s)$ be a vector space of \emph{morphisms} between $r$ and $s$. Let us have an associative binary operation $\otimes\colon R\times R\to R$ and associative bilinear operations $\otimes\colon\Mor(r,s)\times\Mor(r',s')\to\Mor(r\otimes r',s\otimes s')$. Let $\cdot:\Mor(r,s)\otimes\Mor(p,r)\to\Mor(p,s)$ be another associative bilinear operation. Finally, let $*$ be an antilinear involution mapping $\Mor(r,s)\to\Mor(s,r)$. Then the tuple $(R,\{\Mor(r,s)\}_{r,s\in R},\otimes,\cdot,*)$ forms a (small strict) \emph{monoidal $*$-category} if the following additional conditions hold:
\begin{itemize}
\item For every $r\in R$, there is an identity $1_r\in\Mor(r,r)$ satisfying $1_r\cdot T_1=T_1$ and $T_2\cdot 1_r=T_2$ for every $T_1\in\Mor(p,r)$ and every $T_2\in\Mor(r,s)$.
\item There is $1\in R$ such that, for every $r\in R$, $1\otimes r=r\otimes 1=r$.
\end{itemize}

A monoidal $*$-category is called \emph{concrete} if the morphisms are realized by matrices. That is, there is a map $n\colon R\to\N_0$ such that $\Mor(r,s)\subset\Lin(\C^{n(r)},\C^{n(s)})$, the identity morphism is the identity matrix and the operations coincide with the classical operations on matrices.

Let $R_1, R_2$ be monoidal $*$-categories. A map $F\colon R_1\to R_2$ together with linear maps $F\colon\Mor(r,s)\to\Mor(F(r),F(s))$ are called a \emph{monoidal unitary functor} if they preserve the structure of a monoidal $*$-category, i.e.
\begin{itemize}
\item $F(T_1\otimes T_2)=F(T_1)\otimes F(T_2)$,
\item $F(T_2T_1)=F(T_2)F(T_1)$,
\item $F(T^*)=F(T)^*$.
\end{itemize}
We will often consider functors that act on the objects as some obvious bijection. In this case, we will call the linear maps $F\colon\Mor(r,s)\to\Mor(F(r),F(s))$ a \emph{monoidal $*$-homomorphism}. If those are also bijections (i.e.\ linear isomorphisms), we will call them \emph{monoidal $*$-isomorphism}.

Sometimes we will refer to a \emph{monoidal involutive categories}, by which we will mean a monoidal $*$-category, where we drop the linear structure of the morphism spaces.

For given two objects $r,s$ of a monoidal category, we will say that they are \emph{dual} to each other (denoted $s=\bar r$ or $r=\bar s$) if there are morphisms $T_1\in\Mor(1,r\otimes s)$ and $T_2\in\Mor(1,s\otimes r)$ such that $(T_1^*\otimes 1_r)(1_r\otimes T_2)=1_r$ and $(T_2^*\otimes 1_s)(1_s\otimes T_1)=1_s$. A monoidal category, where all objects have their dual is called a monoidal category \emph{with duals}.

\subsection{Tannaka--Krein duality} An important example of a concrete monoidal $*$-category with duals is the set of all unitary representations $\Rep G$ of a given compact matrix quantum group $G$, where the set of morphisms between two representations $u$ and $v$ is the space of intertwiners $\Mor(u,v)$. A dual of a representation $u=(u_{ij})$ is simply its complex conjugate $\bar u=(u_{ij}^*)$. (In this paper we work only with Kac type quantum groups, where $\bar u$ is indeed unitary. Otherwise it has to be unitarized.) Such a category is, in addition, \emph{complete} in the sense that it is closed under taking equivalent objects, subobjects and direct sums of objects.

For every compact matrix quantum group $G=(C(G),u)$, it was shown \cite{Wor87} that all representations are direct sums of irreducible ones and that any irreducible representation $v$ is contained as a subrepresentation in a tensor product of sufficiently many copies of the fundamental representation $u$ and its complex conjugate $\bar u$. Thus, to describe the representation theory of a given quantum group $G$, it is enough to consider the category $\Reptil G$ of representations that are made as a tensor product of copies of $u$ and $\bar u$. The complete category $\Rep G$ can be computed as the natural completion of $\Reptil G$.

One of the most important results is the Tannaka--Krein duality for compact quantum groups that was proven by Woronowicz in \cite{Wor88}. It says that conversely given a concrete monoidal $*$-category $R$ generated by some object $r$ and its complex conjugate $\bar r$ there exists a compact matrix quantum group $G$ such that $\Rep G$ is the completion of $R$.

This quantum group is determined uniquely in the sense that the Hopf $*$-algebra $\C[G]$ generated by the matrix elements $u_{ij}$ of the fundamental representation is unique. For the C*-algebra $C(G)$, one can take any C*-completion of $\C[G]$. In this work, we will always consider the \emph{full algebra}, that is, the universal C*-enveloping algebra of $\C[G]$.

\section{Linear categories of partitions}

In 2009 Banica and Speicher introduced the notion of easy quantum groups \cite{BS09}, which allows to construct examples of compact matrix quantum groups from so-called categories of partitions. In this section, we generalize this approach by introducing linear combinations of partitions and extending the usual categorical operations for them. For more information about the connection between partitions and quantum groups, see also the survey \cite{Web17}.

\subsection{Partitions}

Let $k,l\in\N_0$, by a \emph{partition} of $k$ upper and $l$ lower points we mean a partition of the set $\{1,\dots,k\}\sqcup\{1,\dots,l\}\approx\{1,\dots,k+l\}$, that is, a decomposition of the set of $k+l$ points into non-empty disjoint subsets, called \emph{blocks}. The first $k$ points are called \emph{upper} and the last $l$ points are called \emph{lower}. The set of all partitions on $k$ upper and $l$ lower points is denoted $\Lart(k,l)$. We denote the union $\Lart:=\bigcup_{k,l\in\N_0}\Lart(k,l)$. The number $\left| p\right|:=k+l$ for $p\in\Lart(k,l)$ is called the \emph{length} of $p$.

We illustrate partitions graphically by putting $k$ points in one row and $l$ points on another row below and connecting by lines those points that are grouped in one block. All lines are drawn between those two rows.

Below, we give an example of two partitions $p\in \Lart(3,4)$ and $q\in\Lart(4,4)$ defined by their graphical representation. The first set of points is decomposed into three blocks, whereas the second one is into five blocks. In addition, the first one is an example of a \emph{non-crossing} partition, i.e.\ a partition that can be drawn in a way that lines connecting different blocks do not intersect (following the rule that all lines are between the two rows of points). On the other hand, the second partition has one crossing. 

\begin{equation}
\label{eq.pq}
\vrule height 16bp depth 10bp width 0bp
p=
\BigPartition{
\Pblock 0 to 0.25:2,3
\Pblock 1 to 0.75:1,2,3
\Psingletons 0 to 0.25:1,4
\Pline (2.5,0.25) (2.5,0.75)
}
\qquad
q=
\BigPartition{
\Psingletons 0 to 0.25:1,4
\Psingletons 1 to 0.75:1,4
\Pline (2,0) (3,1)
\Pline (3,0) (2,1)
\Pline (2.75,0.25) (4,0.25)
}
\end{equation}

In our graphical notation, if two or more strings cross each other, we never assume they are connected. On the other hand, if three strings meet at one point (typically like this \Partition{\Pline (1,0.2)(1,0.9) \Pline(0.5,0.9)(1.5,0.9)}), we of course assume they are connected. Thus, a partition on two upper and two lower points, where all points are in a single block, is denoted like this $\connecterpart$, whereas the diagram $\crosspart$ stands for a partition consisting of two blocks.

A block containing a single point is called a \emph{singleton}. In particular, the partitions containing only one point are called singletons and for clarity denoted by an arrow $\singleton\in\Lart(0,1)$ and $\upsingleton\in\Lart(1,0)$.

\subsection{Linear categories of partitions}
\label{secc.cat}

Let us fix a natural number $N\in\N$. Let us denote $\Lart\nlin(k,l)$ the vector space of formal linear combination of partitions $p\in\Lart(k,l)$. That is, $\Lart\nlin(k,l)$ is a vector space, whose basis is $\Lart(k,l)$. Let us denote $\Lart\nlin:=\bigcup_{k,l}\Lart\nlin(k,l)$.

Now, we are going to define some operations on $\Lart\nlin$. First, let us define those operations just on partitions.
\begin{itemize}
\item  The \emph{tensor product} of two partitions $p\in\Lart(k,l)$ and $q\in\Lart(k',l')$ is the partition $p\otimes q\in \Lart(k+k',l+l')$ obtained by writing the graphical representations of $p$ and $q$ ``side by side''.
$$
\BigPartition{
\Pblock 0 to 0.25:2,3
\Pblock 1 to 0.75:1,2,3
\Psingletons 0 to 0.25:1,4
\Pline (2.5,0.25) (2.5,0.75)
}
\otimes
\BigPartition{
\Psingletons 0 to 0.25:1,4
\Psingletons 1 to 0.75:1,4
\Pline (2,0) (3,1)
\Pline (3,0) (2,1)
\Pline (2.75,0.25) (4,0.25)
}
=
\BigPartition{
\Pblock 0 to 0.25:2,3
\Pblock 1 to 0.75:1,2,3
\Psingletons 0 to 0.25:1,4,5,8
\Psingletons 1 to 0.75:5,8
\Pline (2.5,0.25) (2.5,0.75)
\Pline (6,0) (7,1)
\Pline (7,0) (6,1)
\Pline (6.75,0.25) (8,0.25)
}
$$

\item For $p\in\Lart(k,l)$, $q\in\Lart(l,m)$ we define their \emph{composition} $qp\in\Lart\nlin(k,m)$ by putting the graphical representation of $q$ below $p$ identifying the lower row of $p$ with the upper row of $q$. The upper row of $p$ now represents the upper row of the composition and the lower row of $q$ represents the lower row of the composition. Each extra connected component of the diagram that appears in the middle and is not connected to any of the upper or the lower points, transforms to a multiplicative factor $N$.
$$
\BigPartition{
\Psingletons 0 to 0.25:1,4
\Psingletons 1 to 0.75:1,4
\Pline (2,0) (3,1)
\Pline (3,0) (2,1)
\Pline (2.75,0.25) (4,0.25)
}
\cdot
\BigPartition{
\Pblock 0 to 0.25:2,3
\Pblock 1 to 0.75:1,2,3
\Psingletons 0 to 0.25:1,4
\Pline (2.5,0.25) (2.5,0.75)
}
=
\BigPartition{
\Pblock 0.5 to 0.75:2,3
\Pblock 1.5 to 1.25:1,2,3
\Psingletons  0.5 to  0.75:1,4
\Pline (2.5,0.75) (2.5,1.25)
\Psingletons -0.5 to -0.25:1,4
\Psingletons  0.5 to  0.25:1,4
\Pline (2,-0.5) (3,0.5)
\Pline (3,-0.5) (2,0.5)
\Pline (2.75,-0.25) (4,-0.25)
}
= N^2
\BigPartition{
\Pblock 0 to 0.25:2,3,4
\Pblock 1 to 0.75:1,2,3
\Psingletons 0 to 0.25:1
\Pline (2.5,0.25) (2.5,0.75)
}
$$

\item For $p\in\Lart(k,l)$ we define its \emph{involution} $p^*\in\Lart(l,k)$ by reversing its graphical representation with respect to the horizontal axis.
$$
\left(
\BigPartition{
\Pblock 0 to 0.25:2,3
\Pblock 1 to 0.75:1,2,3
\Psingletons 0 to 0.25:1,4
\Pline (2.5,0.25) (2.5,0.75)
}
\right)^*
=
\BigPartition{
\Pblock 1 to 0.75:2,3
\Pblock 0 to 0.25:1,2,3
\Psingletons 1 to 0.75:1,4
\Pline (2.5,0.25) (2.5,0.75)
}
$$
\end{itemize}

Now we can extend the definition of tensor product and composition on the whole vector space $\Lart\nlin$ linearly. We extend the definition of the involution antilinearly. These operations are called the \emph{category operations} on partitions.

The set of all natural numbers with zero $\N_0$ as a set of objects together with the spaces of linear combinations of partitions $\Lart\nlin(k,l)$ as sets of morphisms between $k\in\N_0$ and $l\in\N_0$ with respect to those operations form a monoidal $*$-category. All objects in the category are self-dual.

Any collection of subspaces $\Kat=\bigcup_{k,l\in\N_0}\Kat(k,l)$, $\Kat(k,l)\subset\Lart\nlin(k,l)$ containing the \emph{identity partition} $\idpart\in\Kat(1,1)$ and the \emph{pair partition} $\pairpart\in\Kat(0,2)$ and closed under the category operations is a monoidal $*$-category with duals. We call it a \emph{linear category of partitions}.

For given $p_1,\dots,p_n\in\Lart\nlin$ we denote by $\langle p_1,\dots,p_n\rangle\nlin$ the smallest linear category of partitions containing $p_1,\dots,p_n$. We say that $p_1,\dots,p_n$ \emph{generate} $\langle p_1,\dots,p_n\rangle\nlin$. Note that the pair partitions are contained in the category by definition and hence will not be explicitly listed as generators.

The category operations on partitions were first defined by Banica and Speicher in \cite{BS09}. They also defined the notion of a category of partitions. Nonetheless, their definition does not fully coincide with ours since we consider linear combinations of partitions here. For the comparison, see also Subsection \ref{secc.easy}.

Note that if we consider two partitions $p,q\in\Lart(k,k)$, where all blocks are of size two (so-called pairings), then the composition $qp$ in $\Lart\nlin$ coincides (including the factors $N$) with the multiplication in the Brauer algebra $B_k(N)$ (defined by Brauer in \cite{Bra37}). Considering only non-crossing pairings, we get the Temperley--Lieb algebra $TL_k(N)$ (originally defined in \cite{TL71}; the multiplication was interpreted as a composition of pairings in \cite{Kau87}).

\subsection{Partitions with lower points only}
\label{secc.plow}

For $p\in\Lart(k,l)$, $k>0$, its \emph{left rotation} is a partition $\Lrot p\in\Lart(k-1,l+1)$ obtained by moving the leftmost point of the upper row to the beginning of the lower row, while it still belongs to the same block as before. Similarly, for $p\in\Lart(k,l)$, $l>0$, we can define its \emph{right rotation} $\Rrot p\in\Lart(k+1,l-1)$ by moving the last point of the lower row to the end of the upper row. Both operations are obviously invertible.


For linear combinations of partitions we define those operations linearly.

\begin{lem}
Any linear category of partitions is closed under rotations (left and right including inverses).
\end{lem}
\begin{proof}
The proof for linear combination of partitions is the same as for partitions. See \cite[Lemma 2.7]{BS09}.
\end{proof}

Since any linear category of partitions $\Kat$ is closed under rotations, it means that it is completely described by its subset consisting of partitions with lower points only. When working with partitions with lower points only, it is convenient to introduce the \emph{rotation} as a map $p\mapsto Rp:=(\Lrot\circ\Rrot)p$ for $p\in\Lart\nlin(0,k)$ which takes the last point of a partition and moves it to the front.
%
%
%
%
%

\subsection{Linear maps associated to partitions} Consider again a fixed natural number $N\in\N$. Given a partition $p\in\Lart(k,l)$, we can define a linear map $T_p\colon(\C^N)^{\otimes k}\to(\C^N)^{\otimes l}$ via
\begin{equation}
T_p(e_{i_1}\otimes\cdots\otimes e_{i_k})=\sum_{j_1,\dots,j_l=1}^N\delta_p(\mathbf{i},\mathbf{j})(e_{j_1}\otimes\cdots\otimes e_{j_l}),
\end{equation}
where $\mathbf{i}=(i_1,\dots,i_k)$, $\mathbf{j}=(j_1,\dots,j_l)$ and the symbol $\delta_p(\mathbf{i},\mathbf{j})$ is defined as follows. Let us assign the $k$ points in the upper row of $p$ by the numbers $i_1,\dots,i_k$ (from left to right) and the $l$ points in the lower row by $j_1,\dots,j_l$ (again from left to right). Then $\delta(\mathbf{i},\mathbf{j})=1$ if the points belonging to the same block are assigned the same numbers. Otherwise $\delta(\mathbf{i},\mathbf{j})=0$. 

As an example, we can express $\delta_p$ and $\delta_q$, where $p$ and $q$ come from Equation \eqref{eq.pq}, using multivariate $\delta$ function as follows
$$\delta_p(\mathbf{i},\mathbf{j})=\delta_{i_1i_2i_3j_2j_3},\quad
\delta_q(\mathbf{i},\mathbf{j})=\delta_{i_2j_3j_4}\delta_{i_3j_2}.$$

We extend this definition for linear combinations of partitions linearly, i.e.\ $\delta_{\alpha p+q}=\alpha\delta_p+\delta_q$ and hence $T_{\alpha p+q}=\alpha T_p+T_q$.

Given a linear combination of partitions $p\in\Lart\nlin(k,l)$, we can interpret the map $T_p$ as an intertwiner $T_pu^{\otimes k}=u^{\otimes l}T_p$ for some compact matrix quantum group $G$. Substituting the definition of $T_p$, this implies the following relations
$$\sum_{t_1,\dots,t_k=1}^N\delta_p(\mathbf{t},\mathbf{s})u_{t_1i_1}\cdots u_{t_ki_k}=\sum_{j_1,\dots,j_l=1}^N\delta_p(\mathbf{i},\mathbf{j})u_{s_1j_1}\cdots u_{s_lj_l}$$
for every $i_1,\dots,i_k,s_1,\dots,s_l\in\{1,\dots,N\}$.

For example, considering $p=\pairpart\in\Lart(0,2)$, we have the relation $$\delta_{s_1s_2}=\sum_{j=1}^Nu_{s_1j}u_{s_2j}.$$ Thus, for any quantum group $G\subset O_N^+$, we have that $T_{\pairpart}\in\Mor(1,u\otimes u)$. Similarly, we also have $T_{\uppairpart}\in\Mor(u\otimes u,1)$ for any $G\subset O_N^+$.

\begin{prop}
The map $T_\bullet\colon p\mapsto T_p$ is a monoidal $*$-homomorphism. That is, we have the following
\begin{enumerate}
\item $T_{p\otimes q}=T_p\otimes T_q$,
\item $T_{qp}=T_qT_p$ whenever one of the sides makes sense,
\item $T_{p^*}=T_p^*$.
\end{enumerate}
\end{prop}

Note in particular the distinction with the original work \cite{BS09}, where the composition is considered without the factors $N$ and therefore the map $T_\bullet$ is not a functor. If we consider just partitions and not their linear combinations, it does not make a difference. However, in case of linear combinations of partitions, it is essential to include the factors $N$ in the definition of composition to assure that $T_\bullet$ indeed is a functor.

Note also that $T_\bullet$ is not injective. Indeed, consider for example $N=2$. Then we have
$$\delta_{\Labc}=1=\delta_{\Labb}+\delta_{\Laab}+\delta_{\Laba}-2\delta_{\Laaa},$$
so
$$T_{\Labc}=T_{\Labb}+T_{\Laab}+T_{\Laba}-2T_{\Laaa},$$
To be more precise, we can formulate the following proposition. Recall, that $T_\bullet$ depends on the number $N$. For a partition $p\in\Lart(k,l)$, denote by $b(p)$ the number of blocks in $p$.

\begin{prop}
For any $k,l\in\N_0$, the set
$$\{T_p\mid p\in\Lart(k,l);\;b(p)\le N\}$$
forms a basis of $\spanlin\{T_p\mid p\in\Lart(k,l)\}=\{T_p\mid p\in\Lart\nlin(k,l)\}$.
\end{prop}
\begin{proof}
We will prove the lemma using yet an alternative basis of $\spanlin\{T_p\}_{p\in\Lart(k,l)}$. In \cite[Definition 3.1]{Maa18}, Maassen introduced alternative maps $\hat T_p\colon\C^{Nk}\to\C^{Nl}$ associated to partitions $p\in\Lart(k,l)$. As remarked in the article, it is clear from their definition that $\hat T_p=0$ if $b(p)>N$. In addition, it is proven \cite[Lemma 3.4]{Maa18} that the rest, i.e.\ the set $\{\hat T_p\mid p\in\Lart(k,l);\;b(p)\le N\}$, is linearly independent.

In addition, the following relationship between $T_p$ and $\hat T_p$ can be formulated \cite[Lemma 4.21]{Maa18}
$$T_p=\sum_{q\ge p}\hat T_q,$$
where we write $q\ge p$ if $q$ was obtained by joining some blocks in $p$. This proves that $\spanlin\{T_p\}_{p\in\Lart(k,l)}=\spanlin\{\hat T_p\}_{p\in\Lart(k,l)}$ and also that the set $\{T_p\mid p\in\Lart(k,l);\;b(p)\le N\}$ is a basis since it is obtained by a regular transformation of the basis $\{\hat T_p\mid p\in\Lart(k,l);\;b(p)\le N\}$.
\end{proof}

\begin{cor}
\label{C.Tiso}
Let $k,l\in\N_0$ such that $k+l\le N$. Then $T_\bullet$ restricted to $\Lart(k,l)$ is injective.
\end{cor}
\begin{proof}
The number of blocks in a given partition is always lower or equal to the number of points. Since we assume that the number of points $k+l\le N$, it follows from the previous proposition that the operators $\{T_p\mid p\in\Lart(k,l)\}$ are linearly independent.
\end{proof}

\subsection{Quantum groups associated to linear categories}
\begin{thm}[{\cite{BS09}}]
Let us denote by $u$ the fundamental representation of the group $S_N$. Then we have
$$\Mor(u^{\otimes k},u^{\otimes l})=\spanlin\{T_p\mid p\in\Lart(k,l)\}=\{T_p\mid p\in\Lart\nlin(k,l)\}.$$
\end{thm}

Using this theorem together with the Tannaka--Krein duality, we get the following corollary.

\begin{cor}
\label{C.Tannaka}
For any compact matrix quantum group $G=(C(G),u)$ such that $S_N\subset G\subset O_N^+$, there is a linear category of partitions $\Kat$ such that $\Reptil G$ is the image of $\Kat$ by the homomorphism $T_\bullet$. That is,
	$$\Mor(u^{\otimes k},u^{\otimes l})=\{T_p\mid p\in\Kat(k,l)\}.$$
Conversely, for any linear category of partitions $\Kat$ there is a quantum group $G$, $S_N\subset G\subset O_N^+$, whose representation theory corresponds to $\Kat$ in this way.
\end{cor}

We can express the compact matrix quantum group $G$ corresponding to a linear category of partitions $\Kat$ very concretely as
$$C(G)=C^*(u_{ij},\;i,j=1,\dots,N\mid u=\bar u,\; T_pu^{\otimes k}=u^{\otimes l}T_p\;\forall p\in\Kat(k,l),\;k,l\in\N_0).$$
Suppose that a set $K$ generates $\Kat$. Then, thanks to the functorial property of $T$, we have
$$C(G)=C^*\left(u_{ij},\;i,j=1,\dots,N\mathbin\bigg|\begin{matrix} u=\bar u,\; uu^t=u^tu=1_{\C^N},\\ T_pu^{\otimes k}=u^{\otimes l}T_p\;\forall p\in K(k,l),\;k,l\in\N_0\end{matrix}\right).$$

\subsection{Easy categories}
\label{secc.easy}
A linear category of partitions $\Kat$ is called \emph{easy} if it spanned by partitions (not linear combinations of partitions). That is, if there exists a collection of sets $\Cat(k,l)\subset\Lart(k,l)$ such that $\Kat(k,l)=\spanlin\Cat(k,l)$. Note that in this case the set $\Cat=\bigcup_{k,l\in\N_0}\Cat(k,l)$ is closed under the category operations if we ignore the scalar factors in the composition rule, which are in this case unimportant. Thus, $\Cat$ forms a monoidal involutive category. Such a category is called a \emph{category of partitions} according to the original definition of Banica and Speicher from \cite{BS09}. The corresponding quantum groups $G$ with $S_N\subset G\subset O_N^+$ in the sense of Corollary~\ref{C.Tannaka} are called \emph{easy}; otherwise they are \emph{non-easy}. Note also that since we can ignore the scalar factors in the definition of the composition, easy categories do not depend on the number $N\in\N$ corresponding to the size of the matrux $u$ of the quantum group, while general linear categories of partitions might do.

The easy categories of partitions are classified \cite{RW16} and this classification provides us important examples of linear categories of partitions, so let us briefly mention some results of this classification.

One of the interesting kind of categories of partitions are the non-crossing ones, where all elements are non-crossing partitions. There are exactly the following seven non-crossing easy categories of partitions \cite{Web13}.
$$\langle\rangle\nlin\subset\Big\{\begin{matrix}\langle\singleton\otimes\singleton\rangle\nlin\subset\langle\Labac\rangle\nlin\subset\langle\singleton\rangle\nlin\\\langle\fourpart\rangle\nlin\subset\langle\fourpart,\singleton\otimes\singleton\rangle\nlin\end{matrix}\Big\}\subset\langle\fourpart,\singleton\rangle\nlin$$
We denote the corresponding quantum groups as follows.
$$O_N^+\supset\Big\{\begin{matrix}B_N^{\#+}\supset B_N'^+\supset B_N^+\\H_N^+\supset S_N'^+\end{matrix}\Big\}\supset S_N^+$$
Adding the crossing partition $\crosspart$, which corresponds to the commutativity relation $u_{ij}u_{kl}=u_{kl}u_{ij}$, to those categories, we obtain six categories (only six because $\langle\singleton\otimes\singleton,\crosspart\rangle\nlin=\langle\Labac,\crosspart\rangle\nlin$) corresponding to groups
$$O_N\supset\Big\{\begin{matrix}B_N\times\Z_2\supset B_N\\H_N\supset S_N\times\Z_2\end{matrix}\Big\}\supset S_N.$$
Here, $O_N$ denotes the orthogonal group, $B_N$ is the bistochastic group, $H_N$ stands for the hyperoctahedral group and $S_N$ is the symmetric group. This also motivates the notation for the \emph{free} quantum groups corresponding to the non-crossing categories.

\subsection{The role of the double singleton $\singleton\otimes\singleton$}
In this article, we are interested particularly in the categories containing $\singleton\otimes\singleton$, so let us comment a bit on those.

The partition $\singleton\otimes\singleton$ is a rotation of $\disconnecterpart$, which corresponds to the relation $\sum_k u_{ik}=\sum_k u_{kj}$ for any $i,j$. That is, sums of all rows and all columns are equal. Let us denote those sums $r:=\sum_k u_{ik}$. Using orthogonality of $u$, we can, in addition, derive that $r^2=1$, which is actually the relation corresponding to $\singleton\otimes\singleton$.

The relation corresponding to $\Pabac$, which is a rotation of $\Labac$, can be written as $ru_{ij}=u_{ij}r$, i.e.\ $r$ commutes with everything. The relation corresponding to the singleton $\singleton$ says that $r=1$.

\subsection{Products of CMQGs}

\begin{prop}[\cite{Wan95tensor}]
Let $G=(C(G),u)$ and $H=(C(H),v)$ be compact matrix quantum groups. Then $G\times H:=(C(G)\otimes_{\rm max}C(H),u\oplus v)$ is a compact matrix quantum group. For the co-multiplication we have that
$$\Delta_\times(u_{ij}\otimes 1)=\Delta_G(u_{ij}),\quad\Delta_\times(1\otimes v_{kl})=\Delta_H(v_{kl}).$$
\end{prop}

The algebra $C(G)\otimes_{\rm max}C(H)$ can be described as a universal C*-algebra generated by elements $u_{ij}$ and $v_{kl}$ such that every $u_{ij}$ commutes with every $v_{kl}$, the elements $u_{ij}$ satisfy the same relations as $u_{ij}\in C(G)$ and the elements $v_{kl}$ satisfy the same relations as $v_{kl}\in C(H)$. Thus, the matrix 
$$u\oplus v=\begin{pmatrix}u&0\\ 0&v\end{pmatrix}\in M_{N_1+N_2}(C(G)\otimes_{\rm max}C(H))$$
indeed consists of generators of the algebra $C(G)\otimes_{\rm max}C(H)$. From now on, we will use such interpretation of the maximal tensor product and we will write just $u_{ij}v_{kl}$ instead of $u_{ij}\otimes v_{kl}$ without the explicit tensor sign.

A similar construction can be defined using the free product.

\begin{prop}[\cite{Wan95free}]
Let $G=(C(G),u)$ and $H=(C(H),v)$ be compact matrix quantum groups. Then $G*H:=(C(G)*_\C C(H),u\oplus v)$ is a compact matrix quantum group. For the co-multiplication we have that
$$\Delta_*(u_{ij})=\Delta_G(u_{ij}),\quad\Delta_*(v_{kl})=\Delta_H(v_{kl}).$$
\end{prop}

We view $C(G)*_{\C} C(H)$ as the universal C*-algebra generated by elements $u_{ij}$ satisfying the relations from $C(G)$ and elements $v_{kl}$ satisfying the relations from $C(H)$, identifying the units but imposing no further relations.

We will call the quantum groups $G\times H$ and $G*H$ the \emph{tensor product} and the \emph{free product} of $G$ and $H$.

Actually, Wang defined those products in his articles in the general setting of compact quantum groups. However, in our article, we understand by the tensor and free product always this particular compact \emph{matrix} quantum group construction.

For example, we will often use the following construction. Consider a compact matrix quantum group $G=(C(G),u)$ and denote by $E:=(\C,1)$ the trivial compact matrix (quantum) group. Then we can construct $G\times E=(C(G),u\oplus 1)$, which is isomorphic to $G$ in the sense that there is a $*$-isomorphism mapping $C(G\times E)=C(G)\to C(G)$, but it is not identical with $G$ in the sense of the definition formulated in Subsection \ref{secc.qgdef} since the fundamental representation of $G$ is an $N\times N$ matrix, while the one of $G\times E$ is an $(N+1)\times (N+1)$ matrix.

\begin{prop}[\cite{TW17}] Let $G=(C(G),u)$ and $H=(C(H),v)$ be compact matrix quantum groups. Let $A$ be the C*-subalgebra of $C(G)\otimes_{\rm max} C(H)$ generated by the products $u_{ij}v_{kl}$, i.e.\ generated by the elements of the matrix $u\otimes v$. Then $G\tiltimes H:=(A,u\otimes v)$ is a compact matrix quantum group. For the co-multiplication we have that
$$\Delta_{\tiltimes}(u_{ij}v_{kl})=\Delta_G(u_{ij})\Delta_H(v_{kl})$$
\end{prop}

The quantum group $G\tiltimes H$ is called the \emph{glued tensor product} of $G$ and $H$. Similarly, one can define the \emph{glued free product} $G\tilstar H$, which we will not use in this article.

As an example, let us mention that $S_N'^+=S_N^+\tiltimes\hat\Z_2$ and $B_N'^+=B_N^+\tiltimes\hat\Z_2$ \cite[Proposition 5.1]{Web13}.

\section{Constructing a quantum group generated by a subrepresentation}
\label{sec.isometry}
In this section, we introduce the main objects of this article -- the quantum groups $G_+^{\rm irr}$ and $G_-^{\rm irr}$ associated to a quantum group $G$ with $S_N\subset G\subset B_N^{\# +}$. The starting point is Lemma~\ref{L.reducible}. See also Subsection \ref{secc.warmup} for a motivation.

\subsection{Quantum groups with reducible subrepresentation}
The fundamental representation of $S_N$ decomposes into a direct sum of two irreducible representations: the trivial representation acting on the invariant subspace spanned by the vector $\xi:=\sum_{i=1}^N e_i$ and the \emph{standard representation}, which acts faithfully on the orthogonal complement of $\spanlin\{\xi\}$.

Therefore, the fundamental representation of any quantum group $G\supset S_N$ has at most those two invariant subspaces.

\begin{lem}[{\cite[Proposition 2.5(iii)]{RW15}}]
\label{L.reducible}
Let $G$ be a compact matrix quantum group such that $S_N\subset G\subset O_N^+$ corresponding to a linear category of partitions $\Kat$. The fundamental representation of $G$ is reducible if and only if $\singleton\otimes\singleton\in\Kat$, which holds if and only if $G\subset B_N^{\#+}$.
\end{lem}
\begin{proof}
As mentioned above, the fundamental representation $u$ of a quantum group $G\supset S_N$ is reducible if and only if $\spanlin\{\xi\}$ is an invariant subspace. The projection onto $\spanlin\{\xi\}$ can be written as $\frac{1}{N}T_{\disconnecterpart}$. Thus, $u$ is reducible if and only if $T_{\disconnecterpart}\in\Mor(u,u)$, which holds if and only if $\disconnecterpart\in\Kat$, which holds if and only if $\singleton\otimes\singleton\in\Kat$. (Recall that $B_N^{\#+}$ is the easy quantum group, whose category is generated by the partition $\singleton\otimes\singleton$.)
\end{proof}

Consider a quantum group $G$ such that $S_N\subset G\subset B_N^{\# +}$, so its fundamental representation $u$ has two invariant subspaces -- $\spanlin\{\xi\}$ and its orthogonal complement $\spanlin\{\xi\}^\perp$. This means that taking any linear map $U\colon\C^N\to\C^N$ such that $\spanlin\{\xi\}^\perp$ is mapped onto the space spanned by the first $N-1$ basis vectors $\spanlin\{e_1,\dots,e_{N-1}\}$ and $\xi$ is mapped onto (a multiple of) $e_N$, we get that $UuU^{-1}=v\oplus r$, where $v\in M_{N-1}(C(G))$ and $r\in C(G)$. 

If, in addition, the matrix $U$ is orthogonal, then $UuU^{-1}$ is orthogonal, which means that $v$ is orthogonal and $r$ is a self-adjoint unitary (i.e.\ $r=r^*$ and $r^2=1$). In particular, both $v$ and $r$ are unitary representations of $G$. To extract just the subrepresentation $v$, we can define an $(N-1)\times N$ matrix $V$ by taking the first $N-1$ rows of $U$. Then we have $v=VuV^*$.

Note that in the condition $U\xi=\alpha e_N$ the orthogonality implies $\alpha=\pm\sqrt{N}$. The condition $U(\spanlin\{\xi\}^\perp)\subset\spanlin\{e_1,\dots,e_{N-1}\}$ is then satisfied automatically. Equivalently, we may require that the last row of $U$ equals to $\frac{\pm 1}{\sqrt{N}}\xi^*$.

For the rest of this subsection, suppose that $U\in M_N(\C)$ is an orthogonal matrix such that $U\xi=\pm\sqrt{N}e_N$ and $V$ is the $(N-1)\times N$ matrix obtained by taking the first $N-1$ rows of $U$. 

\begin{lem}
\label{L.V}
$V^*$ is an isometry and $\ker V=\spanlin\{\xi\}$. That is, $VV^*=1_{\C^{N-1}}$ and $V^*V=P_{(N)}$, where $P_{(N)}$ is the orthogonal projection onto $\spanlin\{\xi\}^\perp$.
\end{lem}
\begin{proof}
The matrix $V$ can be expressed as $V=EU$, where $E$ is the ``standard'' coisometry $\C^N\to\C^{N-1}$ mapping $e_i\mapsto e_i$ for $i<N$ and $e_N\mapsto 0$. So, we have
$$VV^*=EUU^*E=EE^*=1_{\C^{N-1}}.$$
From this, it already follows that $V^*V$ is a projection. Its range is $V^*V\C^N=V^*\C^{N-1}$, so it is spanned by the rows of $V$ and hence it is indeed the orthogonal complement of the last row of $U$, which is a multiple of $\xi$.
\end{proof}

From the block structure $UuU^*=v\oplus r$, it follows that $\Delta(v_{ij})=\sum_k v_{ik}\otimes v_{kj}$, so we can define the following.

\begin{defn}
\label{D.VGV}
Let $G=(C(G),u)$, $u=(u_{ij})_{i,j=1}^N$ be a compact matrix quantum group such that $S_N\subset G\subset B_N^{\# +}$. Then we denote $VGV^*:=(A,v)$, where $v=VuV^*\in M_{N-1}(C(G))$ and $A$ is the $C^*$-subalgebra of $C(G)$ generated by $\{v_{ij}\}_{i,j=1}^{N-1}$.
\end{defn}

\begin{lem}
\label{L.VTV}
Let $G=(C(G),u)$, $u=(u_{ij})_{i,j=1}^N$ be a compact matrix quantum group such that $S_N\subset G\subset B_N^{\# +}$ and denote $v$ the fundamental representation of $VGV^*$. Then its intertwiner spaces are
$$\Mor(v^{\otimes k},v^{\otimes l})=\{V^{\otimes l}TV^{*\,\otimes k}\mid T\in\Mor(u^{\otimes k},u^{\otimes l})\}.$$
\end{lem}
\begin{proof}
For the inclusion $\supset$ take an arbitrary $T\in\Mor(u^{\otimes k},u^{\otimes l})$, so $u^{\otimes l}T=Tu^{\otimes k}$. Now, we have
\begin{align*}
(V^{\otimes l}u^{\otimes l}V^{*\,\otimes l})(V^{\otimes l}TV^{*\,\otimes k})&=V^{\otimes l}u^{\otimes l}TV^{*\,\otimes k}\\&=V^{\otimes l}Tu^{\otimes k}V^{*\,\otimes k}=(V^{\otimes l}TV^{*\,\otimes k})(V^{\otimes k}u^{\otimes k}V^{*\,\otimes k}),
\end{align*}
where we used that $V^*V$ is the projection onto the $(N-1)$-dimensional invariant subspace, so it commutes with $u$. For the inclusion $\subset$, we can similarly prove that given an intertwiner $T\in\Mor(v^{\otimes k},v^{\otimes l})$, we have that $V^{*\,\otimes l}TV^{\otimes k}\in\Mor(u^{\otimes k},u^{\otimes l})$ and we can express $T=V^{\otimes k}(V^{*\,\otimes k}TV^{\otimes l})V^{*\,\otimes l}$ using that $VV^*$ is the identity.
\end{proof}

\subsection{The definition of $G_+^{\rm irr}$ and $G_-^{\rm irr}$}
In the following, we define such an orthogonal matrix $U$ explicitly. Recall that we require the last row of $U$ to be $\frac{\pm 1}{\sqrt{N}}\xi^*$. For the rest of the matrix $U$ we can choose arbitrary rows that complete the last one to an orthonormal basis. Nevertheless, we choose a very specific symmetric form, where the matrix elements $U_{ij}$ can be written as a combination $a\delta_{ij}+b$. The motivation will be clear in the following text (compare also with Section \ref{sec.similar}).

\begin{defn}
\label{D.U}
Let us define two orthogonal matrices $U_{(N,+)},U_{(N,-)}\in M_N(\C)$ as follows
$$[U_{(N,\pm)}]_{ij}=\delta_{ij}-{1\over N-1}\left(1\pm{1\over\sqrt{N}}\right),$$
$$[U_{(N,\pm)}]_{iN}=[U_{(N,\pm)}]_{Nj}=[U_{(N,\pm)}]_{NN}=\pm{1\over\sqrt{N}}.$$
for $i,j\in\{1,\dots,N-1\}$. Let us also denote $V_{(N,\pm)}$ the $(N-1)\times N$ matrix formed by the first $N-1$ rows of $U_{(N,\pm)}$.
\end{defn}
%
%

\begin{defn}
\label{D.Girr}
	Let $G=(C(G),u)$ be a compact matrix quantum group such that $S_N\subset G\subset B_N^{\# +}$. Then we denote $G^{\rm irr}_\pm:=V_{(N,\pm)}GV_{(N,\pm)}^*$ as in Definition \ref{D.VGV}.
\end{defn}

\begin{lem}
\label{L.VSV}
It holds that
$$S_{N-1}\subset V_{(N,\pm)}S_{N}V_{(N,\pm)}^*.$$
\end{lem}
\begin{proof}
Take any permutation $\sigma\in S_N$ such that $\sigma(N)=N$. Consider its permutation matrix $A\in S_N$ given by $A_{ij}=\delta_{i\sigma(j)}$. Now, it is enough to check that $VAV^*=A|_{\C^{N-1}}$, where, for simplicity, we write just shortly $V:=V_{(N,\pm)}$. Indeed, thanks to the fact that $V$ is of the form $V_{ij}=a\delta_{ij}+b+c\delta_{Nj}$, it is easy to see that $V_{i\sigma^{-1}(j)}=V_{\sigma(i)j}$, so
$$[VAV^*]_{ij}=\sum_{k,l}V_{ik}\delta_{k\sigma(l)}V_{jl}=\sum_{k}V_{ik}V_{j\sigma^{-1}(k)}=\sum_k V_{ik}V_{\sigma(j)k}=[VV^*]_{i\sigma(j)}=\delta_{i\sigma(j)}.\qedhere$$
\end{proof}

We now prove the first part of our main Theorem~\ref{T1}.

\begin{thm}
\label{T.VG}
Let $G$ be a compact matrix quantum group such that $S_N\subset G\subset B_N^{\# +}$. Then it holds that $S_{N-1}\subset (S_N)^{\rm irr}_\pm\subset G^{\rm irr}_\pm\subset O_{N-1}^+$.
\end{thm}
\begin{proof}
For any pair of quantum groups $G,H$ such that $S_N\subset H\subset G\subset B_N^{\#+}$ (subgroups in the sense of Subsection \ref{secc.qgdef}), we have $H^{\rm irr}_{\pm}\subset G^{\rm irr}_{\pm}$ since the corresponding $*$-homomorphism $C(G)\to C(H)$ must also map the corresponding $(N-1)$-dimensional subrepresentations. From this, the middle inclusion follows.

Since the matrix $V_{(N,\pm)}$ is a partial isometry according to Lemma~\ref{L.V}, we have that the orthogonality of $u$ implies the orthogonality of $V_{(N,\pm)}uV_{(N,\pm)}^*$. So, we have $G^{\rm irr}_\pm\subset O_{N-1}^+$. Finally, the inclusion $S_{N-1}\subset (S_N)^{\rm irr}_\pm$ is due to Lemma \ref{L.VSV}.
\end{proof}

\begin{rem}
The quantum groups $G^{\rm irr}_+$ and $G^{\rm irr}_-$ are by construction obviously similar. In Section \ref{sec.similar}, we present an explicit regular matrix $T$ such that $TG^{\rm irr}_+ T^{-1}=G^{\rm irr}_-$ (see Proposition~\ref{P.tauirr}).
\end{rem}

\subsection{Induced map of partitions}
\label{secc.VP}
From Theorem~\ref{T.VG} and Corollary~\ref{C.Tannaka} it follows that for any quantum group corresponding to a linear category of partitions $\Kat\owns\singleton\otimes\singleton$, we have that the quantum group $G^{\rm irr}_\pm$ also corresponds to some category $\Kat_\pm$. Our goal is to describe this category explicitly.

In Lemma~\ref{L.VTV} we have proven that the intertwiner spaces for the quantum group $G^{\rm irr}_\pm$ are of the form
$$\Mor(v^{\otimes k},v^{\otimes l})=\{V_{(N,\pm)}^{\otimes l}T_pV_{(N,\pm)}^{*\,\otimes k}\mid p\in\Kat\}.$$
Thus, it remains to find for all linear combinations of partitions $p\in\Lart\nlin$ a linear combination $q\in\Lart\nnnlin$ such that $T_q=V_{(N,\pm)}^{\otimes l}T_pV^{*\,\otimes k}_{(N,\pm)}$. We will now define an operator $\V_{(N,\pm)}$ mapping $p\mapsto q$.

To formulate it we need to introduce some notation. A partition, where all points are elements of a single block is called a \emph{block partition}. A block partition with no upper points and $k$ lower points will be denoted $b_k\in\Lart(0,k)$.

Now, consider a subset $I\subset\{1,\dots,k\}$. A partition made by disconnecting all points in the set $I$ from a block partition $b_k$ will be denoted $b_{k,I}$. For example, $b_{4,\{2,4\}}$ is a partition with no upper and four lower points, where on position 2 and 4 there is a singleton and the rest is contained in one block, so $b_{4,\{2,4\}}=\Labac$.

Now, let us denote
$$b_{k,i}:=\sum_{\substack{I\subset\{1,\dots,k\}\\ |I|=i}} b_{k,I}$$
for all $i<k$. Note that this definition implies $b_{k,k-1}=k\singleton^{\otimes k}$. We have for example
$$b_{3,0}=b_3=\Laaa,\quad b_{3,1}=\Laab+\Laba+\Labb,\quad b_{3,2}=3\,\Labc.$$

Finally, let us define in $\Lart\nnnlin(0,k)$ the element $b_{k,k}:=(N-1)\singleton^{\otimes k}$.

\begin{defn}
\label{D.VK}
We define the following operator $\V_{(N,\pm)}\colon\Lart\nlin\to\Lart\nnnlin$. For a block partition $b_k\in\Lart(0,k)$, we define
$$\V_{(N,\pm)}b_k=\sum_{i=0}^k \left({-1\over N-1}\left(1\pm{1\over\sqrt{N}}\right)\right)^ib_{k,i}+\left({\pm1\over\sqrt{N}}\right)^k\singleton^{\otimes k}\in\Lart\nnnlin(0,k).$$
For a block partition with $k$ upper and $l$ lower points $\Rrot^k b_{k+l}\in\Lart(k,l)$ we define $\V_{(N,\pm)}\Rrot^k b_{k+l}:=\Rrot^k\V_{(N,\pm)}b_{k+l}$. For general $p\in\Lart$, we define the action of $\V_{(N,\pm)}$ blockwise. For general $p\in\Lart\nlin$ we extend the action linearly.
\end{defn}

\begin{ex}
\label{ex.V}
As an example, let us mention how $\V_{(N,\pm)}$ acts on the smallest block partitions.
\begin{align*}
&\V_{(N,\pm)}b_1=\V_{(N,\pm)}\singleton=0,\\
&\V_{(N,\pm)}b_2=\V_{(N,\pm)}\pairpart=\pairpart,\\
&\V_{(N,\pm)}b_3=\V_{(N,\pm)}\Laaa\\&=\Laaa-{1\over N-1}\left(1\pm{1\over\sqrt{N}}\right)(\Laab+\Laba+\Labb)+{1\over (N-1)^2}\left(2\pm{N+1\over\sqrt{N}}\right)\Labc,\\
&\V_{(N,\pm)}b_4=\Laaaa-{1\over N-1}\left(1\pm{1\over\sqrt{N}}\right)(\Laaab+\Laaba+\Labaa+\Labbb)+\\&+{1\over (N-1)^2}\left(\frac{N+1}{N}\pm{2\over\sqrt{N}}\right)(\Laabc+\Labac+\Labca+\\&+\Labbc+\Labcb+\Labcc)+\frac{1}{(N-1)^3}\left(\frac{N^2-6N-5}{N}\mp\frac{8}{\sqrt{N}}\right)\Labcd.
\end{align*}
\end{ex}

\begin{lem}
\label{L.VV}
The map $\V_{(N,\pm)}$ satisfies the following.
\begin{enumerate}
\item For any $p\in\Lart\nlin$, we have $\V_{(N,\pm)}p=p+q$, where $q$ is a linear combination of partitions containing a singleton.
\item $\ker\V_{(N,\pm)}|_{\Lart\nlin(k,l)}=\spanlin\{p\in\Lart(k,l)\mid\hbox{$p$ contains a singleton}\}$.
\end{enumerate}
\end{lem}
\begin{proof}
Directly from the definition of $\V_{(N,\pm)}$ we see that (1) holds for block partitions. Since $\V_{(N,\pm)}$ acts blockwise, it must hold for any partition $p$. 

To prove part (2) note that we have $\V\singleton=0$ (see the preceding Example \ref{ex.V}). Since $\V$ acts blockwise, we have also $\V p=0$ for any $p$ containing a singleton. This proves the inclusion $\supset$. For the opposite inclusion, we use part (1). Consider $p\in\ker\V_{(N,\pm)}$, so $0=\V_{(N,\pm)}p=p+\vphantom q$partitions containing a singleton. Thus, $p$ must be a linear combination of partitions containing a singleton.
\end{proof}

We complete the proof of our main Theorem~\ref{T1} by describing the category of partitions corresponding to the quantum groups $G_\pm^{\rm irr}$.

\begin{thm}
\label{T.VK}
It holds that
$$T_{\V_{(N,\pm)}p}=V_{(N,\pm)}^{\otimes l}T_pV_{(N,\pm)}^{*\,\otimes k}$$
for any $p\in\Lart\nlin(k,l)$. Hence, if $G$ is a compact matrix quantum group corresponding to a category $\Kat$ containing $\singleton\otimes\singleton$, then $G_\pm^{\rm irr}$ corresponds to
$$\V_{(N,\pm)}\Kat=\{\V_{(N,\pm)}p\mid p\in\Kat(k,l)\}\subset\Lart\nnnlin.$$
\end{thm}
\begin{proof}
To simplify the notation in the proof, let us denote $V:=V_{(N,\pm)}$ and $\V:=\V_{(N,\pm)}$.

To prove the first part, it is enough to show for $p\in\Lart\nlin(k,l)$ that
\begin{equation}
\label{eq.dVp}
\delta_{\V p}(\mathbf{i},\mathbf{j})=\sum_{t_1,\dots,t_k=1}^{N}\sum_{s_1,\dots,s_l=1}^{N}V_{i_1t_1}\cdots V_{i_kt_k}V_{j_1s_1}\cdots V_{j_ls_l}\delta_p(\mathbf{t},\mathbf{s})
\end{equation}
since then we can compute
\begin{align*}
&T_{\V p}(e_{i_1}\otimes\cdots\otimes e_{i_k})=\sum_{j_1,\dots,j_l=1}^{N-1}\delta_{\V p}(\mathbf{i},\mathbf{j})(e_{j_1}\otimes\cdots\otimes e_{j_l})\\
&=\sum_{j_1,\dots,j_l=1}^{N-1}\sum_{t_1,\dots,t_k=1}^{N}\sum_{s_1,\dots,s_l=1}^{N}V_{i_1t_1}\cdots V_{i_kt_k}V_{j_1s_1}\cdots V_{j_ls_l}\delta_p(\mathbf{t},\mathbf{s})(e_{j_1}\otimes\dots\otimes e_{j_l})\\
&=V^{\otimes l}T_p V^{*\,\otimes k}(e_{i_1}\otimes\cdots\otimes e_{i_k}).
\end{align*}

Since $\V$ acts blockwise, it is enough to prove Equation \eqref{eq.dVp} for block partitions with no upper points $b_l\in\Lart(0,l)$. First, note that
$$\sum_{s=1}^{N-1}\sum_{\substack{\{\alpha_1,\dots,\alpha_{k-i}\}\\\subset\{1,\dots,l\}}}\delta_{s=j_{\alpha_1}=\cdots=j_{\alpha_{l-i}}}=\delta_{b_{l,i}}(\mathbf{j})$$
for all $i=0,\dots,k$ and for all $j_1,\dots,j_{l}\in\{1,\dots,N-1\}$. Note in particular the case $i=l$, where the left-hand side equals to
$$\sum_{s=1}^{N-1}1=N-1=(N-1)\delta_{\singleton^{\otimes l}}(\mathbf{j})=\delta_{b_{l,l}}(\mathbf{j}),$$
so the factor $N-1$ indeed appears.

Now, we can compute for any $j_1,\dots,j_k\in\{1,\dots,N-1\}$
\begin{align*}
&\sum_{s_1,\dots,s_l=1}^{N}V_{j_1s_1}\cdots V_{j_ls_l}\delta_{b_l}(0,\mathbf{s})=
\sum_{s=1}^{N}V_{j_1s}\cdots V_{j_ls}=\\
&=\sum_{s=1}^{N-1} V_{j_1s}\cdots V_{j_l s}+V_{j_1N}\cdots V_{j_l N}=\\
&=\sum_{s=1}^{N-1}\prod_{\alpha=1}^l\left(\delta_{j_\alpha s}-\frac{1}{N-1}\left(1\pm\frac{1}{\sqrt{N}}\right)\right)+\left(\frac{\pm 1}{\sqrt{N}}\right)^l=\\
&=\sum_{s=1}^{N-1}\sum_{i=0}^l\sum_{\substack{\{\alpha_1,\dots,\alpha_{l-i}\}\\\subset\{1,\dots,l\}}}\left(-\frac{1}{N-1}\left(1\pm\frac{1}{\sqrt{N}}\right)\right)^i\delta_{s=j_{\alpha_1}=\dots=j_{\alpha_{l-i}}}+\left(\frac{\pm 1}{\sqrt{N}}\right)^l=\\
&=\sum_{i=0}^l\left(-\frac{1}{N-1}\left(1\pm\frac{1}{\sqrt{N}}\right)\right)^i\delta_{b_{l,i}}(\mathbf{j})+\left(\frac{\pm 1}{\sqrt{N}}\right)^l=\delta_{\V b_l}(\mathbf{j}).
\end{align*}

Finally, from Lemma~\ref{L.VTV} it follows that indeed the set $\V_{(N,\pm)}\Kat$ corresponds to the quantum group $G_\pm^{\rm irr}=(C(G_\pm^{\rm irr}),v)$ in the sense that $\Mor(v^{\otimes k},v^{\otimes l})=\{T_q\mid q\in\V_{(N,\pm)}\Kat\}$.
\end{proof}

\begin{rem}
Since the functor $T_\bullet$ mapping $q\mapsto T_q$ is not injective, it might be that $\V_{(N,\pm)}\Kat$ is not a category. However, from Proposition~\ref{P.Viso} it will follow that $\V_{(N,\pm)}\Kat$ is indeed a linear category of partitions.
\end{rem}

\begin{rem}
\label{R.Vfunct}
The map $\V_{(N,\pm)}$ is not a functor. We have for example
$$(\V_{(N,\pm)}(\upsingleton\otimes\idpart\otimes\idpart))\cdot(\V_{(N,\pm)}(\singleton\otimes\pairpart))=0\cdot 0=0,$$
$$\V_{(N,\pm)}((\upsingleton\otimes\idpart\otimes\idpart)\cdot(\singleton\otimes\pairpart))=\V_{(N,\pm)}\pairpart=\pairpart.$$
\end{rem}

\begin{rem}
Let us try to give some explanation for the explicit form of the map $\V_{(N,\pm)}$. We give an alternative definition of this map, which is less convenient for practical computations, but it is maybe a bit simpler to understand, and sketch the corresponding proof of Theorem~\ref{T.VK}.

First, note that the map $V_{(N,\pm)}$ can be written as a composition $V_{(N,\pm)}=T_\upsilon B$, where $\upsilon=\idpart-\frac{1}{N-1}\left(1\pm\frac{1}{\sqrt{N}}\right)\disconnecterpart\in\Lart\nnnlin(1,1)$, so $T_\upsilon\in M_{N-1}(\C)$, and $B$ is a $(N-1)\times N$ matrix with entries $B_{ij}=\delta_{ij}-\delta_{Nj}$.

Now, we can define the map $\V_{(N,\pm)}$ in two steps as well. First, one can prove that for any linear combination of partitions $p\in\Lart\nlin(k,l)$ we have that $B^{\otimes l}T_pB^{*\,\otimes k}=T_{\mathcal{B}p}$, where $\mathcal{B}p\in\Lart\nnnlin(k,l)$ was made from $p$ by replacing every block $b_j$ by the linear combination $b_j-\singleton^{\otimes j}$. As a consequence, we have that $V^{\otimes l}T_pV^{*\,\otimes k}=T_{\V p}$, where we define $\V p:=\upsilon^{\otimes l}(\mathcal{B}p)\upsilon^{\otimes k}$. Noticing that $\upsilon\singleton=\mp\frac{1}{\sqrt{N}}\singleton$, it is easy to check that this definition of $\V$ indeed coincides with Definition \ref{D.VK}. 
\end{rem}

\begin{ex}
\label{ex.Bn}
As an immediate application of Theorem~\ref{T.VK}, let us consider the case $G=B_N^+$. This quantum group corresponds to the category $\Kat=\langle\singleton\rangle\nlin$ spanned by all non-crossing partitions with blocks of size at most two. Since $\V_{(N,\pm)}$ acts blockwise, we can see from Example \ref{ex.V} that it acts as the identity for all pair partitions (partitions with all blocks of size two). On the other hand, as we mentioned in Lemma~\ref{L.VV}, any partition containing singleton is mapped to zero. Thus, we have that $\V_{(N,\pm)}\Kat$ is the category spanned by all non-crossing pair partitions. According to Theorem \ref{T.VK}, the quantum group $G_\pm^{\rm irr}$ corresponds to the category $\V_{(N,\pm)}\Kat$, so it must be the free orthogonal quantum group $O_{N-1}^+$.
\end{ex}

\section{The subrepresentation within the $N$-dimensional setting}
\label{sec.projection}

\subsection{Projection instead of the partial isometry}
Recall the notation $P_{(N)}$ for the orthogonal projection onto the $(N-1)$-dimensional invariant subspace $\spanlin\{\xi\}^\perp$. Consider a quantum group $G=(C(G),u)$ such that $S_N\subset G\subset B_N^{\# +}$. In the previous section, we defined the quantum group $G^{\rm irr}_\pm=(C(G^{\rm irr}_\pm),v)$ generated by the $(N-1)$-dimensional subrepresentation. Now, it may be useful to go back to the dimension $N$ and study the matrix $V_{(N,\pm)}^*vV_{(N,\pm)}=P_{(N)}uP_{(N)}$.

Note that the matrix $P_{(N)}uP_{(N)}$ does not generate a compact matrix quantum group $P_{(N)}GP_{(N)}$ according to our definition since it is not invertible. Nevertheless, it can be seen as a quantum group isomorphic to $G^{\rm irr}_\pm$ acting on the $(N-1)$-dimensional subspace $P_{(N)}\C^N\subset\C^N$.

We are going to use the linear combinations of partitions to study the following category.

\begin{defn}
We denote by $\Reptil P_{(N)}GP_{(N)}$ the monoidal $*$-category with the set of natural numbers with zero, $\N_0$, as objects and the following sets as morphisms
$$\Mor((P_{(N)}uP_{(N)})^{\otimes k},(P_{(N)}uP_{(N)})^{\otimes l}):=\{P_{(N)}^{\otimes l}TP_{(N)}^{\otimes k}\mid T\in\Mor(u^{\otimes k},u^{\otimes l})\}$$
between given objects $k$ and $l$. The operations are defined in the standard way. For every object $k\in\N_0$ there is the identity morphism $P^{\otimes k}$.
\end{defn}

\begin{prop}
\label{P.repiso}
The category $\Reptil PGP$ is monoidally $*$-isomorphic to $\Reptil G^{\rm irr}_\pm$ via
$$P_{(N)}^{\otimes l}TP_{(N)}^{\otimes k}\mapsto V_{(N,\pm)}^{\otimes l}P_{(N)}^{\otimes l}TP_{(N)}^{\otimes k}V_{(N,\pm)}^{*\,\otimes k}=V_{(N,\pm)}^{\otimes l}TV_{(N,\pm)}^{*\,\otimes k}.$$
\end{prop}
\begin{proof}
We just make use of Lemma~\ref{L.V} saying that $V_{(N,\pm)}V_{(N,\pm)}^*=1_{\C^{N-1}}$ and that $V_{(N,\pm)}^*V_{(N,\pm)}=P_{(N)}$.

The map is indeed bijective since we can express the inverse as
$$V_{(N,\pm)}^{\otimes l}TV_{(N,\pm)}^{*\,\otimes k}\mapsto V_{(N,\pm)}^{*\,\otimes l}V_{(N,\pm)}^{\otimes l}TV_{(N,\pm)}^{*\,\otimes k}V_{(N,\pm)}^{\otimes k}=P_{(N)}^{\otimes l}TP_{(N)}^{\otimes k}.$$

Checking the functoriality is also straightforward. Let us check for example the composition:
\begin{align*}
(P_{(N)}^{\otimes m}T_2P_{(N)}^{\otimes l})(P_{(N)}^{\otimes l}T_1P_{(N)}^{\otimes k})\mapsto &V_{(N,\pm)}^{\otimes l}P_{(N)}^{\otimes m}T_2P_{(N)}^{\otimes l}T_1P_{(N)}^{\otimes k}V_{(N,\pm)}^{*\,\otimes k}\\&=V_{(N,\pm)}^{\otimes m}T_2V_{(N,\pm)}^{*\,\otimes l}V_{(N,\pm)}^{\otimes l}T_1V_{(N,\pm)}^{*\,\otimes k}.\qedhere
\end{align*}
\end{proof}

So, on one hand, the category $\Reptil PGP$ provides an alternative description of the quantum group $G^{\rm irr}$. On the other hand note that $P_{(N)}$ is an intertwiner of $u$, so the category $\Reptil PGP$ forms a subset of $\Reptil G$. Adding the proper intertwiners to this category, we are able to reconstruct the category of representations of $G$ and hence the whole quantum group and we are also able to produce new examples of quantum groups.

Our goal now is to transform those considerations into the simpler setting of linear combinations of partitions.

\subsection{The projective partition $\pi_{(N)}$}
\label{secc.pi}

\begin{defn}
Let us define the following linear combination
$$\pi_{(N)}:=\idpart-{1\over N}\disconnecterpart\in\Lart\nlin(1,1).$$
\end{defn}

This partition is \emph{projective} in the sense that we have $\pi_{(N)}\pi_{(N)}=\pi_{(N)}$ and $\pi_{(N)}^*=\pi_{(N)}$ (the definition of a projective partition originally appeared in \cite{FW16}). Therefore, also the map $T_{\pi_{(N)}}$ must be an orthogonal projection.

\begin{lem}
\label{L.Ppi}
$T_{\pi_{(N)}}$ equals to $P_{(N)}$ the orthogonal projection onto $\spanlin\{\xi\}^\perp$.
\end{lem}
\begin{proof}
We have that $\frac{1}{N}T_{\disconnecterpart}=\frac{1}{N}T_{\singleton} T_{\upsingleton}=\frac{1}{N}\xi\xi^*$ is the orthogonal projection onto $\spanlin\{\xi\}$. Therefore, $T_{\pi_{(N)}}=1-\frac{1}{N}T_{\disconnecterpart}$ is the projection onto the orthogonal complement.
\end{proof}

\begin{defn}
\label{D.P}
We define a linear map $\PP_{(N)}\colon\Lart\nlin\to\Lart\nlin$ acting as $\PP_{(N)}p=\pi_{(N)}^{\otimes l}p\pi_{(N)}^{\otimes k}$ for all $p\in\Lart\nlin(k,l)$.
\end{defn}

\begin{ex}
\label{ex.P}
As an example, let us compute the action of $\PP_{(N)}$ on small block partitions (cf. Example \ref{ex.V}):
\begin{align*}
\PP_{(N)}\singleton&=0,\\
\PP_{(N)}\pairpart&=\pairpart-\frac{1}{N}\singleton\otimes\singleton=\Lrot\pi_{(N)},\\
\PP_{(N)}\Laaa&=\Laaa-\frac{1}{N}(\Laab+\Laba+\Labb)+\frac{2}{N^2}\Labc,\\
\PP_{(N)}\Laaaa&=\Laaaa-{1\over N}(\Laaab+\Laaba+\Labaa+\Labbb)+\\&+{1\over N^2}(\Laabc+\Labac+\Labca+\Labbc+\Labcb+\Labcc)+\frac{3}{N^3}\Labcd.
\end{align*}
\end{ex}

\begin{rem}
\label{R.P}
Let us make the following remarks on the action of $\PP_{(N)}$.
\begin{enumerate}
\item The operator $\PP_{(N)}$ acts (similarly as $\V_{(N)}$) by cutting legs from blocks. In particular, for any $p\in\Lart\nlin$ we have that $\PP_{(N)}p=p+q$, where $q$ is a linear combination of partitions containing at least one singleton.
\item For any linear combination $p$ of partitions containing a singleton, we have $\PP_{(N)}p=0$. This follows from the fact that $\pi_{(N)}\singleton=0$.
\item Using the same counterexample as in Remark \ref{R.Vfunct}, we can prove that $\PP_{(N)}$ is not a functor.
\end{enumerate}
\end{rem}

If we consider a linear category $\Kat$ such that $\singleton\otimes\singleton\in\Kat$, then $\PP_{(N)}\Kat\subset\Kat$. The set $\PP_{(N)}$ corresponds to intertwiners of the form $T_{\PP_{(N)}p}=P^{\otimes l}T_pP^{\otimes k}$, i.e.\ elements of the category $\Reptil PGP$. In the following subsection, we are going to describe the sets $\PP_{(N)}\Kat$ in an abstract way giving them also a structure of a category.

\subsection{Reduced linear categories of partitions}

\begin{defn}
\label{D.red}
A vector subspace $\Kat\red\subset\Lart\nlin$ is called \emph{reduced linear category of partitions} if
\begin{itemize}
\item $\pi_{(N)},\Lrot\pi_{(N)},\Rrot\pi_{(N)}\in\Kat\red$,
\item $\PP_{(N)}\Kat\red=\Kat\red$,
\item $\Kat\red$ is closed under the category operations (i.e.\ tensor product, composition, and involution).
\end{itemize}
\end{defn}

Any reduced linear category indeed forms a monoidal $*$-category. The identity morphism is $\pi_{(N)}^{\otimes k}\in\Lart\nlin(k,k)$ for every $k\in\N_0$.

\begin{lem}
Any reduced linear category of partitions is closed under left and right rotation.
\end{lem}
\begin{proof}
The proof is the same as in the case of ordinary categories of partitions. We just have to use $\pi_{(N)}$, $\Lrot\pi_{(N)}$ and $\Rrot\pi_{(N)}$ instead of $\idpart$, $\pairpart$ and $\uppairpart$.
\end{proof}

\begin{rem}
\label{R.red}
From Remark \ref{R.P}(2) it follows that every reduced category $\Kat\red$ does not contain any linear combination of the form $q\otimes\singleton$. In particular, we have $\singleton\otimes\singleton\not\in\Kat\red$ for every reduced category $\Kat\red$. On the other hand, if we complete it to an ordinary linear category of partitions, we always have $\singleton\otimes\singleton\in\langle\Kat\red\rangle\nlin$ since $\singleton\otimes\singleton=N(\pairpart-\Lrot\pi_{(N)})$.
\end{rem}

\begin{prop}
\label{P.Pred}
Let $\Kat$ be a linear category of partitions such that $\singleton\otimes\singleton\in\Kat$. Then $\PP_{(N)}\Kat$ is a reduced category.
\end{prop}
\begin{proof}
Since $\idpart\in\Kat$, we have $\pi_{(N)}=\PP_{(N)}\idpart\in\PP_{(N)}\Kat$ and similarly for its rotations.

From the projective property of $\pi_{(N)}$ it follows that $\PP_{(N)}p=p$ for every $p\in\PP_{(N)}\Kat$, which implies that the second axiom holds true.

Moreover, we also have $\pi_{(N)}^{\otimes l}p=p$ and $p\pi_{(N)}^{\otimes k}=p$ for $p\in\PP_{(N)}\Kat(k,l)$, which can be used to prove the third axiom. For example, taking any $p\in\PP_{(N)}\Kat(k,l)$ and $q\in\PP_{(N)}\Kat(l,m)$, we have $qp\in\Kat(k,m)$ since $\PP_{(N)}\Kat\subset\Kat$ and hence also
\[qp=\pi_{(N)}^{\otimes m}qp\pi_{(N)}^{\otimes k}=\PP_{(N)}(qp)\in\PP_{(N)}\Kat(k,m).\qedhere\]
\end{proof}

In particular, the whole set $\PP_{(N)}\Lart\nlin$ forms a reduced category. This allows us to introduce the notation $\langle p_1,\dots,p_n\rangle\nred$ for the smallest reduced category containing $p_1,\dots,p_n\in\PP_{(N)}\Lart\nlin$.

\begin{lem}
\label{L.rednlin}
Let $p_1,\dots,p_n\in\PP_{(N)}\Lart\nlin$. Then
$$\langle\langle p_1,\dots,p_n\rangle\nred\rangle\nlin=\langle p_1,\dots,p_n,\singleton\otimes\singleton\rangle\nlin.$$
\end{lem}
\begin{proof}
On the left-hand side there is a linear category containing $p_1,\dots,p_n,$ and $\singleton\otimes\singleton$, which implies the inclusion $\supset$. The category on the right hand side contains $p_1,\dots,p_n$ and rotations of $\pi_{(N)}$ and it is of course closed under the category operations. So, $\langle p_1,\dots,p_n\rangle\nred\subset\langle p_1,\dots,p_n,\singleton\otimes\singleton\rangle\nlin$, which implies the opposite inclusion $\subset$.
\end{proof}

We have proven that $\PP_{(N)}\Kat$ is a reduced category for any linear category of partitions $\Kat$ containing $\singleton\otimes\singleton$. Now, we prove that all reduced categories are of this form.

\begin{prop}
\label{P.redP}
Let $\Kat\red$ be a reduced category. Then
$$\Kat\red=\PP_{(N)}\langle\Kat\red\rangle\nlin=\PP_{(N)}\langle\Kat\red,\Labac\rangle\nlin=\PP_{(N)}\langle\Kat\red,\singleton\rangle\nlin.$$
\end{prop}
\begin{proof}
Denote by $K$ the set of all $p'\in\Lart\nlin$ such that $p'$ was made by adding singletons to some $p\in\Kat\red$. To be more precise, we can formulate this condition recursively: for any $p\in K$ it holds that either $p\in\Kat\red$ or there is $q\in K$ such that $p$ is some rotation of $q\otimes\singleton$ (including the possibility that $q$ is a multiple of the empty partition, so $p=\alpha\singleton\in K$ and $p=\alpha\upsingleton\in K$).

Now, let us prove that $\spanlin K$ is a linear category of partitions. The identity partition is a linear combination of $\pi_{(N)}\in\Kat\red\subset K$ and $\disconnecterpart\in K$, so it is contained in $\spanlin K$. Similarly the pair partitions are also contained in $\spanlin K$. It is clear that $K$ is closed under tensor product and involution, so let us prove it for the composition. Take arbitrary composable $p',q'\in K$, which were made from $p,q\in\Kat\red$ by adding singletons. If the added singletons in the lower row of $p'$ do not exactly match the singletons in the upper row of $q'$, we have $q'p'=0$ since $\pi_{(N)}\singleton=0$. Otherwise it is easy to see that $q'p'$ can be made from $qp$ by adding singletons and multiplying by some factor $N^\alpha$, so $q'p'\in\spanlin K$.

This implies that
$$\Kat\red\subset\langle\Kat\red\rangle\nlin\subset\langle\Kat\red,\Labac\rangle\nlin\subset\langle\Kat\red,\singleton\rangle\nlin\subset\spanlin K.$$

Now, since $\PP_{(N)}p=0$ for any $p$ containing a singleton, we have $\PP_{(N)}K=\PP_{(N)}\Kat\red=\Kat\red$. So, applying $\PP_{(N)}$ on the chain of containments above, we have
$$\Kat\red\subset\PP_{(N)}\langle\Kat\red\rangle\nlin\subset\PP_{(N)}\langle\Kat\red,\Labac\rangle\nlin\subset\PP_{(N)}\langle\Kat\red,\singleton\rangle\nlin\subset\Kat\red,$$
which implies the proposition.
\end{proof}
 
\begin{cor}
\label{C.redP}
For any $p_1,\dots,p_n\in\PP_{(N)}\Lart\nlin$, we have
	$$\langle p_1,\dots,p_n\rangle\nred=\PP_{(N)}\langle p_1,\dots,p_n,\singleton\otimes\singleton\rangle\nlin.$$
\end{cor}
\begin{proof}
	Combining Proposition~\ref{P.redP} and Lemma~\ref{L.rednlin}, we have
\[\langle p_1,\dots,p_n\rangle\nred=\PP_{(N)}\langle\langle p_1,\dots,p_n\rangle\nred\rangle\nlin=\PP_{(N)}\langle p_1,\dots,p_n,\singleton\otimes\singleton\rangle\nlin.\qedhere\]
\end{proof}

\subsection{Relation between $\PP\Kat$ and $\V\Kat$}

Consider a linear category of partitions~$\Kat$ with $\singleton\otimes\singleton\in\Kat$ corresponding to a quantum group $G$ with fundamental representation $u$. In the previous sections we constructed the maps between categories as illustrated by the following diagram. From Lemma~\ref{L.V}, it follows that $VP=V$, so the lower part of the diagram commutes. Moreover, in Proposition~\ref{P.repiso} we proved that the map $\Reptil PGP\to\Reptil VGV^*$ is a monoidal $*$-isomorphism. In this subsection we lift those two properties also to the upper part of the diagram.
\begin{center}
\includegraphics{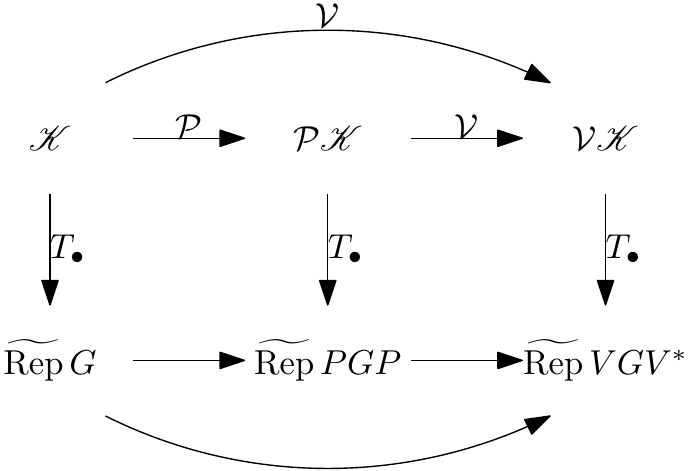}
\end{center}

%

\begin{prop}
\label{P.VP}
It holds that $\V_{(N,\pm)}\PP_{(N)}=\V_{(N,\pm)}$.
\end{prop}
\begin{proof}
Take an arbitrary $p\in\Lart$. By Remark \ref{R.P}(1), we know $\PP_{(N)}p=p+q$, where $q$ is a linear combination of partitions containing a singleton. From Lemma~\ref{L.VV} it follows, that $q\in\ker\V_{(N,\pm)}$, so
\[\V_{(N,\pm)}\PP_{(N)}p=\V_{(N,\pm)}(p+q)=\V_{(N,\pm)}p.\qedhere\]
\end{proof}

As mentioned in Remarks \ref{R.Vfunct} and \ref{R.P}(3), the maps $\V\colon\Kat\to\V\Kat$ and $\PP\colon\Kat\to\PP\Kat$ are not functors. Nevertheless, we prove in the following proposition that, similarly as $\Reptil PGP$ is monoidally $*$-isomorphic to $\Reptil VGV^*$, we have the same isomorphism for $\PP\Kat$ and $\V\Kat$. So, those two categories are also equivalent descriptions of the same phenomenon.

Recall that according to Propositions \ref{P.Pred} and \ref{P.redP}, $\PP\Kat$ is a reduced category and conversely every reduced category is of this form. Note also that the following proposition also proves that $\V_{(N,\pm)}\Kat\red$, resp $\V_{(N,\pm)}\Kat$ is indeed a linear category of partitions for any reduced category $\Kat\red$, resp.\ for any linear category $\Kat\owns\singleton\otimes\singleton$.

\begin{prop}
\label{P.Viso}
Let $\Kat\red\subset\Lart\nlin$ be a reduced category. The map $$\V_{(N,\pm)}\colon\Kat\red\to\V_{(N,\pm)}\Kat\red\subset\Lart\nnnlin$$ is an isomorphism of monoidal $*$-categories.
\end{prop}
\begin{proof}
To simplify the notation, let us denote $V:=V_{(N,\pm)}$, $\V:=\V_{(N,\pm)}$, $P:=P_{(N)}$, $\PP:=\PP_{(N)}$.

First, let us prove the injectivity. From Lemma~\ref{L.VV} and Remark \ref{R.red}, it follows that
$$\ker \V|_{\Kat\red(k,l)}=\Kat\red\cap\spanlin\{p\in\Lart(k,l)\mid\text{$p$ contains a singleton}\}=\{0\}.$$

As for the functorial property note that since $\V$ acts blockwise, it behaves well with respect to the tensor product and involution for arbitrary elements of $\Lart\nlin$. For the composition, take any $p\in\Kat\red(k,l)$ and $q\in\Kat\red(l,m)$, then using Theorem~\ref{T.VK}, Lemma~\ref{L.V}, and Lemma \ref{L.Ppi} we have the following
\begin{align*}
T_{\V q\,\V p}&=T_{\V q}T_{\V p}=V^{\otimes m}T_qV^{*\,\otimes l}V^{\otimes l}T_pV^{*\,\otimes k}=\\&=V^{\otimes m}T_qP^{\otimes l}T_pV^{*\,\otimes k}=V^{\otimes m}T_qT_{\pi^{\otimes l}p}V^{*\,\otimes k}=V^{\otimes m}T_qT_p V^{*\,\otimes k}=T_{\V(qp)}.
\end{align*}
This essentially repeats what we have already proven in Proposition~\ref{P.repiso}. That is, that $\V$ induces a functor $\Reptil PGP\to\Reptil VGV^*$ for the corresponding categories of representations. If the functor $T_\bullet$, $p\mapsto T_p$ was an isomorphism, this would end the proof. Note that the functor $T_\bullet$ depends on $N$. We are going to use Corollary~\ref{C.Tiso} saying that $T_\bullet$ acting on $\Lart\nlin(k,l)$, for fixed $k$ and $l$, is injective for large $N$.

So, let us fix integers $k,l,m\in\N_0$ and two partitions $p\in\Lart(k,l)$ and $q\in\Lart(l,m)$. Let us number all partitions in $\Lart(k,m)$ as $\{x_i\}:=\Lart(k,m)$. Then, we can express for every $N\in\N$ the compositions
\begin{eqnarray*}
\V_{(N,\pm)}q\,\V_{(N,\pm)}p&=:&\sum_i\alpha_i(N)x_i\\
\V_{(N,\pm)}(qp)&=:&\sum_i\beta_i(N)x_i\\
\end{eqnarray*}

As we already computed, we have for every $N\in\N$ the equality $T_{\V_{(N,\pm)}q\,\V_{(N,\pm)}p}=T_{\V_{(N,\pm)}(qp)}$. For $N\ge N_0:=\max\{k,l,m\}$ the functor $T_\bullet$ is an isomorphism, so $\V_{(N,\pm)}q\V_{(N,\pm)}p=\V_{(N,\pm)}(qp)$. In other words, we have $\alpha_i(N)=\beta_i(N)$ for every $i$ and every $N\ge N_0$. From the definition of the map $\V$ and the composition of partitions, it follows that both $\alpha_i(N)$ and $\beta_i(N)$ are polynomials in $\sqrt{N}$ (up to some normalization). If two polynomials of one variable coincide on an infinite set, they must coincide everywhere. Hence, $\alpha_i(N)=\beta_i(N)$ for every $N\in\N$.
\end{proof}

\begin{cor}
\label{C.Vred}
For any $p_1,\dots,p_n\in\PP_{(N)}\Lart\nlin$, we have
$$\V_{(N,\pm)}\langle p_1,\dots,p_n\rangle\nred=\langle\V_{(N,\pm)}p_1,\dots,\V_{(N,\pm)}p_n\rangle\nnnlin.$$
\end{cor}
\begin{proof}
Any element of the left hand side was made by a finite amount of category operations applied on $p_1,\dots,p_n,\pi_{(N)}$ and then applying the functor $\V_{(N,\pm)}$. Thanks to the functoriality of $\V_{(N)}$, this equals to applying the same operations on the elements $\V_{(N,\pm)}p_1,\dots,\V_{(N,\pm)}p_n$, and $\V_{(N,\pm)}\pi_{(N)}=\pairpart$, which are the generators of the right-hand side.
\end{proof}

\begin{cor}
\label{C.VP}
For any $p_1,\dots,p_n\in\Lart\nlin$, we have
$$\V_{(N,\pm)}\langle\PP_{(N)}p_1,\dots,\PP_{(N)}p_n,\singleton\otimes\singleton\rangle\nlin=\langle \V_{(N,\pm)}p_1,\dots,\V_{(N,\pm)}p_n\rangle\nnnlin.$$
\end{cor}
\begin{proof}
Using Proposition~\ref{P.VP} and Corollary~\ref{C.redP} we can express
$$\V_{(N,\pm)}\langle\PP_{(N)}p_1,\dots,\PP_{(N)}p_n,\singleton\otimes\singleton\rangle\nlin=\V_{(N,\pm)}\langle\PP_{(N)}p_1,\dots,\PP_{(N)}p_n\rangle\nred.$$
Now, we just apply the previous corollary.
\end{proof}

\subsection{The way back}

The tools we developed in this section are now going to help us to do the converse of Section \ref{sec.isometry}. That is, to extend an irreducible representation of a quantum group by a one-dimensional representation and describe the intertwiner spaces.

A general question, which we are not going to answer here, is the following. Given a compact quantum group $H$ such that $S_{N-1}\subset (S_N)^{\rm irr}_\pm\subset H\subset O_{N-1}^+$, find all quantum groups $G$ such that $S_N\subset G\subset B_N^{\#+}$ and $H=G^{\rm irr}_\pm$. In the following theorem, constituting our main Theorem~\ref{T2}, we are going to describe the intertwiner spaces of the three canonical choices for such extension $G$. Those are the group $H$ itself extended by a trivial representation, the tensor product $H\times\hat\Z_2$ and the free product $H*\hat\Z_2$.

\begin{thm}
\label{T.redqg}
Let $\Kat\red\subset\Lart\nlin$ be a reduced category. Denote by $H$ the quantum group $S_{N-1}\subset H\subset O_{N-1}^+$ corresponding to the category $\V_{(N,\pm)}\Kat\red$. Then we can construct the quantum group corresponding to the following categories
\begin{eqnarray*}
\langle\Kat\red\rangle\nlin           &\qquad\hbox{corresponds to}\qquad& U_{(N,\pm)}^*(H*\hat\Z_2)U_{(N,\pm)},\\
\langle\Kat\red,\Labac\rangle\nlin    &\qquad\hbox{corresponds to}\qquad& U_{(N,\pm)}^*(H\times\hat\Z_2)U_{(N,\pm)},\\
\langle\Kat\red,\singleton\rangle\nlin&\qquad\hbox{corresponds to}\qquad& U_{(N,\pm)}^*(H\times E)U_{(N,\pm)},\\
\end{eqnarray*}
where $E=(\C,1)$ is the trivial (quantum) group.
\end{thm}
\begin{proof}
For simplicity, denote $U:=U_{(N,\pm)}$, $V:=V_{(N,\pm)}$. Denote by $G$ the quantum group corresponding to $\langle\Kat\red\rangle\nlin$ and by $(u_{ij})_{i,j=1}^n$ its fundamental representation.

As mentioned in Remark \ref{R.red}, $\singleton\otimes\singleton\in\langle\Kat\red\rangle\nlin$, so $UuU^*=v\oplus r$, where $v=VuV^*\in M_{N-1}(C(G))$ and $r=\sum_k u_{ik}\in C(G)$ such that $r^2=1$. Using Proposition~\ref{P.redP} and \ref{P.VP} we derive $\V\langle\Kat\red\rangle\nlin=\V\PP\langle\Kat\red\rangle\nlin=\V\Kat\red$, so we see that $v$ is the fundamental representation of $H$, i.e.\ $H=G_\pm^{\rm irr}$ according to Theorem~\ref{T.VK}. To prove that $G=U^*(H*\hat\Z_2)U$, it remains to show that there are no additional relations in $C(G)$ apart from the relations for $v$ and the relations $r=r^*$, $r^2=1$.

The relations in $C(G)$ are precisely those corresponding to partitions in the category $\langle\Kat\red\rangle\nlin$, which is generated by $\Kat\red$ (and the pair partition $\pairpart$ of course). So, the relations are the orthogonality of $u$, which is equivalent to orthogonality of $v$ and the relations $r=r^*$, $r^2=1$, and the relations implied by the partitions $p\in\Kat\red$. Taking any $p\in\Kat\red$, the relation $T_pu^{\otimes k}=u^{\otimes l}T_p$ is equivalent to
$$U^{\otimes l}T_pU^{*\,\otimes k}(v\oplus r)^{\otimes k}=U^{\otimes l}T_pu^{\otimes k}U^{*\,\otimes k}=U^{\otimes l}u^{\otimes l}T_pU^{*\,\otimes k}=(v\oplus r)^{\otimes l}U^{\otimes l}T_pU^{*\,\otimes k}.$$
Noticing that $T_p=T_{\pi^{\otimes k}p\pi^{\otimes l}}=P^{\otimes k}T_p P^{\otimes l}$ and that $UP=EU$, where $E$ is the orthogonal projection onto the first $N-1$ basis vectors, we see that those relations only contain the subrepresentation $v$ and hence are equivalent to the relations in $C(H)$.

The partitions $\Labac$ and $\singleton$ correspond to additional relations $ru_{ij}=u_{ij}r$ and $r=1$ respectively. From this, the rest of the proposition follows.
\end{proof}

\begin{rem}
As a consequence, we have that the inclusions
$$\langle\Kat\red\rangle\nlin\subset\langle\Kat\red,\Labac\rangle\nlin\subset\langle\Kat\red,\singleton\rangle\nlin,$$
are always strict.
\end{rem}


\section{Examples}
\label{sec.examples}

In this section, we apply the theory developed in the previous sections to easy quantum groups. Given an easy category $\Kat$ containing $\singleton\otimes\singleton$ and the corresponding easy quantum group $G$, we compute the category $\V_{(N,\pm)}\Kat$ corresponding to the quantum group $H:=G^{\rm irr}_\pm$. Then, we apply Theorem~\ref{T.redqg} to compute the intertwiners of the quantum groups $H*\hat\Z_2$, $H\times\hat\Z_2$ and $H\times E$.

\subsection{Non-crossing examples}
In this subsection we are going to use our approach for the easy quantum groups $B_N^+$, $S_N'^+$ and $S_N^+$. We will denote by $NC$ the set of all non-crossing partitions. This is a category in the Banica--Speicher sense corresponding to the quantum group $S_N^+$. We will denote by $NC'$ the set of all non-crossing partitions of even length, which corresponds to the quantum group $S_N'^+=S_N^+\tiltimes\hat\Z_2$. In the language of linear categories of partitions, we have
$$NC\nlin=\langle\fourpart,\singleton\rangle\nlin=\langle\Laaa\rangle\nlin,\qquad NC'\nlin=\langle\fourpart,\singleton\otimes\singleton\rangle\nlin.$$

\begin{lem}
\label{L.Smin}
It holds that
$$\PP_{(N)}NC\nlin=\langle\PP_{(N)}\Laaa\rangle\nred.$$
\end{lem}
\begin{proof}
It is easy to check that $NC\nlin=\langle\Laaa\rangle\nlin=\langle\PP_{(N)}\Laaa,\singleton\rangle\nlin$ (see also Example \ref{ex.P}). From Lemma~\ref{L.rednlin} we have $\langle\langle\PP_{(N)}\Laaa\rangle\nred\rangle\nlin=\langle\PP_{(N)}\Laaa,\singleton\otimes\singleton\rangle\nlin$. Adding the singleton to the category on both sides, we have
$$NC\nlin=\langle\PP_{(N)}\Laaa,\singleton\rangle\nlin=\langle\langle\PP_{(N)}\Laaa\rangle\nred,\singleton\rangle\nlin.$$
Finally, we use Proposition~\ref{P.redP} to derive
\[\PP_{(N)}NC\nlin=\PP_{(N)}\langle\langle\PP_{(N)}\Laaa\rangle\nred,\singleton\rangle\nlin=\langle\PP_{(N)}\Laaa\rangle\nred.\qedhere\]
\end{proof}

\begin{lem}
\label{L.SSmin}
Suppose $N>3$. It holds that
$$\PP_{(N)}NC'\nlin=\langle\PP_{(N)}\Laaaa\rangle\nred.$$
\end{lem}
\begin{proof}
In this case, the inclusion $\supset$ follows from Proposition~\ref{P.Pred}. For the converse, it is easy to see that $\PP_{(N)}NC'\nlin(k,l)=\spanlin\PP_{(N)}NC'(k,l)$. So, it remains to prove $\PP_{(N)}NC'\subset\langle\PP_{(N)}\Laaaa\rangle\nred$. We will show it in four steps.

\medskip\noindent\textsc{Step 1.}\space\textit{$\PP_{(N)}b_k\in\langle\PP_{(N)}\Laaaa\rangle\nred$ for all $k$ even.}

\smallskip\noindent
Here $b_k\in\Lart(0,k)$ is the block partition, that is, a partition where all points are contained in one block (recall the definition in Subsection \ref{secc.VP}). We will show the statement by induction. It holds for $k=2$ by the definition of reduced categories and for $k=4$ since $\PP_{(N)}b_4$ is the generator. Now, consider $k>4$ and suppose $\PP_{(N)} b_i\in\PP_{(N)}\langle\Laaaa\rangle\nred$ for all $i<k$ even. We compute that
\begin{align*}
&(\pi^{\otimes (k-4)}\otimes\PP_{(N)}\connecterpart\otimes\pi^{\otimes 2})(\pi^{\otimes (k-3)}\otimes\PP_{(N)}\Partition{\Pblock 0to0.3:1,2,3 \Pline(1,0.3)(1,1)})\PP_{(N)}b_{k-2}=\\
&=(\pi^{\otimes (k-4)}\otimes\PP_{(N)}\connecterpart\otimes\pi^{\otimes 2})(\PP_{(N)}b_k-\frac{1}{N}\PP_{(N)}b_{k-3}\otimes\PP_{(N)}\Laaa)=\\
&=\left(1-\frac{3}{N}\right)\PP_{(N)}b_k+\frac{1}{N^2}\Big(\PP_{(N)}b_{k-2}\otimes\PP_{(N)}\pairpart+R^{-2}(\PP_{(N)}b_{k-2}\otimes\PP_{(N)}\pairpart)\\&\qquad+\PP_{(N)}b_{k-4}\otimes\PP_{(N)}b_4\Big)-\frac{1}{N^3}\PP_{(N)}b_{k-4}\otimes\PP_{(N)}\pairpart\otimes\PP_{(N)}\pairpart.
\end{align*}
All the terms except for $\PP_{(N)}b_k$ are surely elements of the category $\langle\PP_{(N)}\Laaaa\rangle\nred$, so $\PP_{(N)}b_k$ must be as well.
The idea of the computation is maybe more clear using the pictorial representation for partitions. Using the definition $\PP_{(N)}p=\pi_{(N)}^{\otimes l}p\pi_{(N)}^{\otimes k}$ for $p\in\Lart\nlin(k,l)$ and the projective property $\pi_{(N)}\pi_{(N)}=\pi_{(N)}$, we can express the left hand side in the following form
$$
\BigPartition{
\Pblock 1.2to1.4:1,2,4,5,6
\Ptext (1,1){$\pi$}
\Ptext (2,1){$\pi$}
\Ptext (4,1){$\pi$}
\Ptext (5,1){$\pi$}
\Ptext (6,1){$\pi$}
\Pline (1,0.4) (1,0.8)
\Pline (2,0.4) (2,0.8)
\Pline (4,0.4) (4,0.8)
\Pline (5,0.4) (5,0.8)
\Pline (6,0.4) (6,0.8)
\Pblock 0.4to0.6:6,7,8
\Ptext (1,0.2){$\pi$}
\Ptext (2,0.2){$\pi$}
\Ptext (4,0.2){$\pi$}
\Ptext (5,0.2){$\pi$}
\Ptext (6,0.2){$\pi$}
\Ptext (7,0.2){$\pi$}
\Ptext (8,0.2){$\pi$}
\Pline (1,-0.4) (1,0)
\Pline (2,-0.4) (2,0)
\Pline (4,-0.4) (4,0)
\Pline (5.5,-0.3) (5.5,-0.1)
\Pline (7,-0.4) (7,0)
\Pline (8,-0.4) (8,0)
\Pblock -0.4 to -0.3:5,6
\Pblock 0 to -0.1:5,6
\Ptext (1,-0.6){$\pi$}
\Ptext (2,-0.6){$\pi$}
\Ptext (4,-0.6){$\pi$}
\Ptext (5,-0.6){$\pi$}
\Ptext (6,-0.6){$\pi$}
\Ptext (7,-0.6){$\pi$}
\Ptext (8,-0.6){$\pi$}
\Ptext (3,0.5){$\dots$}
}
=\PP\left(\BigPartition{
\Pblock -0.4to1.4:1,2,4
\Pblock  0.4to1.4:4,5
\Pblock  1.2to1.4:5,6
\Ptext (6,1){$\pi$}
\Pline (6,0.4)(6,0.8)
\Pblock 0.4to0.6:6,7
\Pblock -0.4to0.6:7,8
\Ptext (5,0.2){$\pi$}
\Ptext (6,0.2){$\pi$}
\Pblock 0to-0.1:5,6
\Pblock -0.4to-0.3:5,6
\Pline (5.5,-0.3)(5.5,-0.1)
\Ptext (3,0.5){$\dots$}
}\right).
$$
We obtain the result on the right-hand side simply by substituting $\pi=\idpart-\frac{1}{N}\disconnecterpart$.

\medskip\noindent\textsc{Step 2.}\space\textit{$\PP_{(N)}b_k\otimes\PP_{(N)}b_l\in\langle\PP_{(N)}\Laaaa\rangle\nred$ for any $k,l\ge 2$ such that $k+l$ is even.}

\smallskip\noindent
This can be seen inductively from the following
$$(\pi^{\otimes k}\otimes\PP_{(N)}\uppairpart\otimes\pi^{\otimes l})(\PP_{(N)}b_{k+1}\otimes\PP_{(N)}b_{l+1})=\PP_{(N)}b_{k+l}-\frac{1}{N}(\PP_{(N)}b_k\otimes\PP_{(N)}b_l).$$
Pictorially,
$$
\BigPartition{
\Pblock 0.8to1:1,2,4,5
\Pblock 0.8to1:6,7,9,10
\Ptext(1,0.6){$\pi$}
\Ptext(2,0.6){$\pi$}
\Ptext(3,0.9){$\dots$}
\Ptext(4,0.6){$\pi$}
\Ptext(5,0.6){$\pi$}
\Ptext(6,0.6){$\pi$}
\Ptext(7,0.6){$\pi$}
\Ptext(8,0.9){$\dots$}
\Ptext(9,0.6){$\pi$}
\Ptext(10,0.6){$\pi$}
\Pline(1,0)(1,0.4)
\Pline(2,0)(2,0.4)
\Pline(4,0)(4,0.4)
\Pline(7,0)(7,0.4)
\Pline(9,0)(9,0.4)
\Pline(10,0)(10,0.4)
\Pblock 0.4to0.2:5,6
\Ptext(1,-0.2){$\pi$}
\Ptext(2,-0.2){$\pi$}
\Ptext(3,0.2){$\dots$}
\Ptext(4,-0.2){$\pi$}
\Ptext(7,-0.2){$\pi$}
\Ptext(8,0.2){$\dots$}
\Ptext(9,-0.2){$\pi$}
\Ptext(10,-0.2){$\pi$}
}
=\PP\left({\textstyle\BigPartition{
\Pblock 0to1:1,2,4
\Pblock 0.8to1:4,5
\Pblock 0.8to1:6,7
\Pblock 0to1:7,9,10
\Ptext(3,0.5){$\dots$}
\Ptext(8,0.5){$\dots$}
\Ptext(5,0.6){$\pi$}
\Ptext(6,0.6){$\pi$}
\Pblock 0.4to0.2:5,6
}}\right)
=\PP\left(
\Partition{
\Pblock 0to0.6:1,2,4,5,7,8
\Ptext(3,0.3){$\dots$}
\Ptext(6,0.3){$\dots$}
}-\frac{1}{N}
\Partition{
\Pblock 0to0.6:1,2,4
\Pblock 0to0.6:5,7,8
\Ptext(3,0.3){$\dots$}
\Ptext(6,0.3){$\dots$}
}
\right).
$$

\medskip\noindent\textsc{Step 3.}\space\textit{$(\PP_{(N)}b_k)^*\otimes\idpart\otimes\PP_{(N)} b_k\in\langle\PP_{(N)}\Laaaa\rangle\nred$ for any $k\in\N$ (including the odd ones).}

\smallskip\noindent
For $k=1$ we have $\PP_{(N)}b_k=0$. Considering $k\ge 2$, the assertion follows from
\begin{align*}
&(\pi\otimes(\PP_{(N)}b_k)^*\otimes\PP_{(N)}b_k)((\PP_{(N)}b_k)^*\otimes\PP_{(N)}b_k\otimes\pi)=\\&=\left(\left(1-\frac{1}{N}\right)^{k-1}-\left(\frac{-1}{N}\right)^{k-1}\right)(\PP_{(N)}b_k)^*\otimes\idpart\otimes\PP_{(N)}b_k.
\end{align*}
Pictorially,
$$
\BigPartition{
\Ptext(1,1.2){$\pi$}
\Ptext(2,1.2){$\pi$}
\Ptext(3,1.2){$\pi$}
\Ptext(5,1.2){$\pi$}
\Ptext(6,1.2){$\pi$}
\Ptext(4,1){$\dots$}
\Ptext(4,0.7){$\dots$}
\Pblock 1to0.9:1,2,3,5
\Pblock 0.7to0.8:1,2,3,5
\Pline(6,0.7)(6,1)
\Ptext(1,0.5){$\pi$}
\Ptext(2,0.5){$\pi$}
\Ptext(3,0.5){$\pi$}
\Ptext(5,0.5){$\pi$}
\Ptext(6,0.5){$\pi$}
\Ptext(4,0.3){$\dots$}
\Ptext(4,0.0){$\dots$}
\Pblock 0to0.1:2,3,5,6
\Pblock 0.3to0.2:2,3,5,6
\Pline(1,0)(1,0.3)
\Ptext(1,-0.2){$\pi$}
\Ptext(2,-0.2){$\pi$}
\Ptext(3,-0.2){$\pi$}
\Ptext(5,-0.2){$\pi$}
\Ptext(6,-0.2){$\pi$}
}
=\PP\left(
\BigPartition{
\Ptext(4,1){$\dots$}
\Ptext(4,0.7){$\dots$}
\Pblock 1to0.9:1,2,3,5
\Pblock 0.7to0.8:1,2,3,5
\Pline(6,0.3)(6,1)
\Ptext(2,0.5){$\pi$}
\Ptext(3,0.5){$\pi$}
\Ptext(5,0.5){$\pi$}
\Ptext(4,0.3){$\dots$}
\Ptext(4,0.0){$\dots$}
\Pblock 0to0.1:2,3,5,6
\Pblock 0.3to0.2:2,3,5,6
\Pline(1,0)(1,0.7)
}\right)=\left(\left(1-\frac{1}{N}\right)^{k-1}-\left(\frac{-1}{N}\right)^{k-1}\right)\PP
\Partition{
\Pblock 1to0.5:1,2,3,5
\Pblock 0to0.5:6,7,9,10
\Pline (6,1)(5,0)
\Ptext (4,0.8){\dots}
\Ptext (8,0.2){\dots}
}.
$$

\medskip\noindent\textsc{Step 4.}\space\textit{$\PP_{(N)}p\in\langle\PP_{(N)}\Laaaa\rangle\nred$ for every $p\in NC'$.}

\smallskip\noindent
Without loss of generality we can assume that $p$ has lower points only since reduced categories are closed under rotations. Now, denote $l_1,\dots,l_n$ the sizes of the blocks in $p$ ordered in such a way that first come all even numbers and then all odd numbers. Since $p\in NC'$, we have that $\sum l_i$ is even, so there is an even number of odd numbers in the tuple $(l_i)$. We can construct $\PP_{(N)}p\in\PP_{(N)}\langle\PP_{(N)}\Laaaa\rangle\nred$ by computing $\PP_{(N)}b_{l_1}\otimes\cdots\otimes\PP_{(N)} b_{l_n}\in\langle\PP_{(N)}\Laaaa\rangle\nred$ and then using compositions with
$(\PP_{(N)} b_k)^*\otimes\pi\otimes\PP_{(N)}b_k=
\PP_{(N)}\Partition{
\Pblock 1to0.5:1,2,4,5
\Pblock 0to0.5:6,7,9,10
\Pline (6,1)(5,0)
\Ptext (3,0.8){\dots}
\Ptext (8,0.2){\dots}
}
\in\langle\PP_{(N)}\Laaaa\rangle\nred$
to move the blocks to their positions.
\end{proof}

We are now able to deal with the counterpart of the isomorphism $G=B_N^+\simeq O_{N-1}^+$ for the case of $G=S_N^+$.

\begin{prop}
\label{P.Vqg}
Suppose $N>3$. It holds that
\begin{eqnarray*}
\langle\V_{(N,\pm)}\fourpart\rangle\nnnlin&\text{corresponds to}&(S_N^+\tiltimes\hat\Z_2)^{\rm irr}_\pm=(S_N^+)^{\rm irr}_\pm\tiltimes\hat\Z_2,\\
\langle\V_{(N,\pm)}\Laaa\rangle\nnnlin&\text{corresponds to}&(S_N^+)^{\rm irr}_\pm.
\end{eqnarray*}
\end{prop}
See Example \ref{ex.V} for the explicit form of $\V_{(N,\pm)}\Laaa$ and $\V_{(N,\pm)}\Laaaa$.
\begin{proof}
Using Lemma~\ref{L.Smin}, Proposition~\ref{P.VP} and Corollary~\ref{C.Vred}, we derive
$$\V_{(N,\pm)}NC\nlin=\V_{(N,\pm)}\PP_{(N)}NC\nlin=\V_{(N,\pm)}\langle\PP_{(N)}\Laaa\rangle\nred=\langle\V_{(N,\pm)}\Laaa\rangle\nnnlin,$$
so, according to Theorem~\ref{T.VK}, $\langle\V_{(N,\pm)}\Laaa\rangle\nnnlin$ corresponds to $(S_N^+)^{\rm irr}_\pm$.

Similarly, we deduce that $\langle\V_{(N,\pm)}\Laaaa\rangle\nnnlin$ corresponds to $(S_N^+\tiltimes\hat\Z_2)^{\rm irr}_\pm=V_{(N,\pm)}(S_N^+\tiltimes\hat\Z_2)V_{(N,\pm)}^*$. The quantum group $S_N^+\tiltimes\hat\Z_2$ is determined by the fundamental representation of the form $u'=su$, where $u$ is the fundamental representation of $S_N^+$, $s$ generates $C^*(\Z_2)$ and $su_{ij}=u_{ij}s$. We see that $V_{(N,\pm)}u'V_{(N,\pm)}^*=s(V_{(N,\pm)}uV_{(N,\pm)}^*)=sv$, where $v$ generates the quantum group $(S_N^+)^{\rm irr}_\pm$, so indeed $(S_N^+\tiltimes\hat\Z_2)^{\rm irr}_\pm=(S_N^+)^{\rm irr}_\pm\tiltimes\hat\Z_2$.
\end{proof}

\begin{prop}
\label{p.qg}
The categories
$$\begin{matrix}
\langle \PP_{(N)}\Laaa,\singleton\otimes\singleton\rangle\nlin  & \subset & \langle \PP_{(N)}\Laaa,\Labac\rangle\nlin  & \subset & NC\nlin\\
\subsetrot                                                      &         & \subsetrot                           & & \subsetrot\\
\langle \PP_{(N)}\Laaaa,\singleton\otimes\singleton\rangle\nlin & \subset & \langle \PP_{(N)}\Laaaa,\Labac\rangle\nlin & \subset & \langle \PP_{(N)}\Laaaa,\singleton\rangle\nlin\\
\subsetrot                                                      &         & \subsetrot                           &         & \subsetrot\\
\langle \singleton\otimes\singleton\rangle\nlin                 & \subset & \langle \Labac\rangle\nlin          & \subset & \langle\singleton\rangle\nlin\\
\end{matrix}$$
correspond to quantum groups $G:=U_{(N,\pm)}^*G'U_{(N,\pm)}$, where $G'$ equals to
$$\begin{matrix}
(S_N^+)^{\rm irr}_\pm*\hat\Z_2                       & \supset & (S_N^+)^{\rm irr}_\pm\times\hat\Z_2 & \supset & (S_N^+)^{\rm irr}_\pm\times E\\
\supsetrot                                              &         & \supsetrot                             &         & \supsetrot\\
((S_N^+)^{\rm irr}_\pm\tiltimes\hat\Z_2)*\hat\Z_2 & \supset & ((S_N^+)^{\rm irr}_\pm\tiltimes\hat\Z_2)\times\hat\Z_2 & \supset & ((S_N^+)^{\rm irr}_\pm\tiltimes\hat\Z_2)\times E\\
\supsetrot                                              &         & \supsetrot                             &         & \supsetrot\\
O_{N-1}^+*\hat\Z_2                                & \supset & O_{N-1}^+\times\hat\Z_2          & \supset & O_{N-1}^+\times E\\
\end{matrix}$$
and $E=(\C,(1))$ is the trivial quantum group.
\end{prop}
\begin{proof}
To prove this proposition, we just use Theorem~\ref{T.redqg} for each row. Let us have a look on the first row in more detail. Here, we take $\Kat\red:=\PP_{(N)}NC\nlin=\langle\PP_{(N)}\Laaa\rangle\nred$ (see Lemma~\ref{L.Smin}). According to Proposition~\ref{P.Vqg}, the linear category $\V_{(N,\pm)}\Kat\red=\langle\V_{(N,\pm)}\Laaa\rangle\nnnlin$ corresponds to the quantum group $H:=(S_N^+)^{\rm irr}_\pm$. From Theorem \ref{T.redqg} it follows that the quantum groups
$$\langle\Kat\red,\singleton\otimes\singleton\rangle\nlin\subset\langle\Kat\red,\Labac\rangle\nlin\subset\langle\Kat\red,\singleton\rangle\nlin$$
indeed correspond to the quantum groups given by the first row of the second table. Now, using Lemma~\ref{L.rednlin} we see that $\langle\Kat\red,\singleton\otimes\singleton\rangle\nlin=\langle\PP_{(N)}\Laaa,\singleton\otimes\singleton\rangle\nlin$. Noticing that both $\Labac$ and $\singleton$ generate $\singleton\otimes\singleton$, we can use Lemma \ref{L.rednlin} to prove also $\langle\Kat\red,\Labac\rangle\nlin=\langle\PP_{(N)}\Laaa,\Labac\rangle\nlin$ and $\langle\Kat\red,\singleton\rangle\nlin=\langle\PP_{(N)}\Laaa,\singleton\rangle\nlin$. Finally, it is easy to see that the latter category equals to the lineear category spanned by all non-crossing partitions $NC\nlin$. Indeed, note that $NC\nlin$ is generated by $\Laaa$, which is clearly contained in $\langle\PP_{(N)}\Laaa,\singleton\rangle\nlin$.

The second and third line can be proven by exactly the same argumentation as the first one. This time using Lemma~\ref{L.SSmin} resp. Example \ref{ex.Bn}.
\end{proof}

\begin{rem}
Note that the lower lines of the diagrams in Proposition~\ref{p.qg} reveal again the well-known isomorphisms $B_N^+\simeq O_{N-1}^+$, $B_N'^+\simeq O_{N-1}\times\hat\Z_2$, and $B_N^{\# +}\simeq O_{N-1}*\hat\Z_2$ discovered in \cite[Theorem 4.1]{Rau12} and \cite[Proposition 5.2]{Web13}. See also Example \ref{ex.Bn}.

\end{rem}

\subsection{The crossing partition}

\begin{prop}
\label{P.cross}
We have the following inclusions.
$$\begin{matrix}
\langle\crosspart,\singleton\otimes\singleton\rangle\nlin&=&\langle\crosspart,\Labac\rangle\nlin&\subsetneq&\langle\crosspart,\singleton\rangle\nlin\\
\subsetneqrot&&\eqrot&&\eqrot\\
\langle\PP_{(N)}\crosspart,\singleton\otimes\singleton\rangle\nlin&\subsetneq&\langle\PP_{(N)}\crosspart,\Labac\rangle\nlin&\subsetneq&\langle\PP_{(N)}\crosspart,\singleton\rangle\nlin
\end{matrix}$$
Those categories correspond to quantum groups $G:=U_{(N,\pm)}^*G'U_{(N,\pm)}$, where $G'$ equals to
$$O_{N-1}*\hat\Z_2\supsetneq O_{N-1}\times\hat\Z_2\supsetneq O_{N-1}\times E.$$
\end{prop}
\begin{proof}
The equality and inclusion in the first line is known \cite{Web13} (see also Subsection \ref{secc.easy}). The vertical equalities are easy to see if we write
$$\PP_{(N)}\crosspart=\crosspart-\frac{1}{N}\Pabac-\frac{1}{N}\Pabcb+\frac{1}{N^2}\Pabcd.$$
%

Note that $\V_{(N,\pm)}\PP_{(N)}\crosspart=\V_{(N,\pm)}\crosspart=\crosspart$, so $\V_{(N,\pm)}\langle\PP_{(N)}\crosspart\rangle\nred=\langle\crosspart\rangle\nnnlin$, which corresponds to the (quantum) group $O_{N-1}$. Now, the strictness of the inclusions in the second row as well as the quantum group picture follow from Theorem~\ref{T.redqg} with $\Kat\red:=\langle\PP_{(N)}\crosspart\rangle\nred$.
\end{proof}

Similarly, we can add the partition $\PP_{(N)}\crosspart$ to all the categories in Proposition~\ref{p.qg} and obtain the same quantum groups ``without the plusses''.

The meaning of the element $\PP_{(N)}\crosspart$ can be seen from the corresponding C*-algebraical relation, which is $\left(u_{ij}-\frac{1}{N}r\right)\left(u_{kl}-\frac{1}{N}r\right)=\left(u_{kl}-\frac{1}{N}r\right)\left(u_{ij}-\frac{1}{N}r\right)$. This is equivalent to say that $ab=ba$ for every $a,b\in\spanlin\left\{u_{ij}-\frac{1}{N}r\right\}=\spanlin\{v_{ij}\}$, where $v=V_{(N,\pm)}uV_{(N,\pm)}^*$.

\begin{rem}
Applying $\PP_{(N)}$ to one of the vertical equalities in Proposition~\ref{P.cross} and using Lemma~\ref{L.rednlin} and Proposition \ref{P.redP} as in the proof of Lemma \ref{L.Smin}, we get that
$$\PP_{(N)}\langle\crosspart,\Labac\rangle\nlin=\PP_{(N)}\langle\crosspart,\singleton\rangle\nlin=\langle\PP_{(N)}\crosspart\rangle\nred.$$
Thus, we also have
$$\V_{(N,\pm)}\langle\crosspart,\Labac\rangle\nlin=\V_{(N,\pm)}\langle\crosspart,\singleton\rangle\nlin=\langle\crosspart\rangle\nnnlin.$$
This can be seen also from the quantum group picture. If $\langle\crosspart,\singleton\rangle\nlin$ corresponds to $U_{(N,\pm)}^*(O_{N-1}\times E)U_{(N,\pm)}$ then the map $\V_{(N,\pm)}$ should yield the category corresponding to $O_{N-1}$, which is precisely $\langle\crosspart\rangle\nnnlin$.
\end{rem}

\subsection{Half-liberations}

\begin{prop}
\label{P.hl}
The categories
$$\langle\PP_{(N)}\halflibpart,\singleton\otimes\singleton\rangle\nlin \subsetneq \langle\PP_{(N)}\halflibpart,\Labac\rangle\nlin \subsetneq \langle\PP_{(N)}\halflibpart,\singleton\rangle\nlin$$
correspond to quantum groups $G:=U_{(N,\pm)}^*G'U_{(N,\pm)}$, where $G'$ equals to
$$O_{N-1}^**\hat\Z_2\supsetneq O_{N-1}^*\times\hat\Z_2\supsetneq O_{N-1}^*\times E.$$
\end{prop}
\begin{proof}
Again we have $\V_{(N,\pm)}\PP_{(N)}\halflibpart=\V_{(N,\pm)}\halflibpart=\halflibpart$, so $\V_{(N,\pm)}\langle\PP_{(N)}\halflibpart\rangle\nred=\langle\halflibpart\rangle\nnnlin$, which corresponds to the quantum group $O_{N-1}^*$. Thus, the statement follows from Theorem~\ref{T.redqg}.
\end{proof}

Again, the meaning of the element $\PP_{(N)}\halflibpart$ can be seen from the corresponding C*-algebraical relation, which can be written as $abc=cba$ for every $a,b,c\in\spanlin\left\{u_{ij}-\frac{1}{N}r\right\}=\spanlin\{v_{ij}\}$, where $v=V_{(N,\pm)}uV_{(N,\pm)}^*$.

The category $\langle\PP_{(N)}\halflibpart,\singleton\rangle\nlin$ and the corresponding quantum group was recently independently studied by Banica \cite{Ban18uniform}.

\begin{rem}
In this case, we have no counterpart of the ``vertical equalities'' of Proposition~\ref{P.cross}. In fact, we have
$$\begin{matrix}
\langle \PP_{(N)}\halflibpart,\singleton\otimes\singleton\rangle\nlin  & \subsetneq & \langle \halflibpart,\singleton\otimes\singleton\rangle\nlin  & \subsetneq & \langle\crosspart,\singleton\otimes\singleton\rangle\nlin\\
\supsetneqrot                                                          &            & \supsetneqrot                                                 &            & \eqrot\\
\langle \PP_{(N)}\halflibpart,\Labac\rangle\nlin                       & \subsetneq & \langle \halflibpart,\Labac\rangle\nlin                       & =          & \langle\crosspart,\Labac\rangle\nlin\\
\supsetneqrot                                                          &            & \supsetneqrot                                                 &            & \supsetneqrot\\
\langle \PP_{(N)}\halflibpart,\singleton\rangle\nlin                   & \subsetneq & \langle \halflibpart,\singleton\rangle\nlin                   & =          & \langle\crosspart,\singleton\rangle\nlin
\end{matrix}$$
The inequalities in the last two rows can be seen after applying $\V_{(N,\pm)}$ since we get $\langle\halflibpart\rangle\nnnlin$ on the left-hand side, but $\langle\crosspart\rangle\nnnlin$ on the right-hand side.
To see the inequality in the first row, note for example that $\langle\halflibpart,\singleton\otimes\singleton\rangle\nlin$ contains the linear combination $(\disconnecterpart\otimes\pi\otimes\pi)\Pabcabc$ corresponding to the relation $abr=rba$ for $a,b\in\spanlin\{v_{ij}\}$. Therefore, the quantum group corresponding to $\langle\halflibpart,\singleton\otimes\singleton\rangle\nlin$ contains a relation between $v$ and $r$, so it must be a strict quantum subgroup of the free product $O_{N-1}^**\hat\Z_2$.

Nevertheless, using the explicit description of the category $\langle\halflibpart,\singleton\otimes\singleton\rangle\nlin$ presented in \cite[Proposition 3.5]{Web13}, it is possible to show that
$$\PP_{(N)}\langle\halflibpart,\singleton\otimes\singleton\rangle\nlin=\langle\PP_{(N)}\halflibpart\rangle\nred.$$
\end{rem}

\section{Categories of partitions for similar CMQGs}
\label{sec.similar}

In this section, we apply a procedure analogous to Section \ref{sec.isometry} to construct similar quantum groups. Consider a compact matrix quantum group $G$ such that $S_N\subset G\subset O_N^+$, so $G$ is described by some linear category of partitions $\Kat$. We are going to find a regular matrix $T\in M_N(\C)$ such that $S_N\subset TGT^{-1}\subset O_N^+$, so $\tilde G:=TGT^{-1}$ is also described by some linear category $\tilde\Kat$. We are going to find an explicit description of this category.

%
%
%
%
%

\begin{defn}
We define $\tau_{(N)}:=\idpart-\frac{2}{N}\disconnecterpart\in\Lart\nlin(1,1)$. We define a linear map $\T_{(N)}\colon\Lart\nlin\to\Lart\nlin$ mapping $p\mapsto \tau_{(N)}^{\otimes l}p\tau_{(N)}^{\otimes k}$ for $p\in\Lart\nlin(k,l)$.
\end{defn}

\begin{rem}
\label{R.tau}
Note that it holds that $\tau_{(N)}\tau_{(N)}=\idpart$ and $\tau_{(N)}^*=\tau_{(N)}$. Therefore, the same holds for the map $T_{\tau_{(N)}}$. That is, $T_{\tau_{(N)}}^{-1}=T_{\tau_{(N)}}=T_{\tau_{(N)}}^*$, so $T_{\tau_{(N)}}$ is a self-adjoint unitary.
\end{rem}

\begin{rem}
Note the subtle distinction between the linear combinations $\tau_{(N)}=\idpart-\frac{2}{N}\disconnecterpart$ and $\pi_{(N)}=\idpart-\frac{1}{N}$ that quite changes their properties and their meaning. Note also the similarity between the matrix $T_{\tau_{(N)}}$ having entries $T_{ij}=\delta_{ij}-\frac{2}{N}$ and the matrices $U_{(N+1,\pm)}$, $V_{(N+1,\pm)}$ or $P_{(N)}$.
\end{rem}

\begin{thm}
\label{T.Tiso}
Let $G=(C(G),u)$, $S_N\subset G\subset O_N^+$ be a CMQG corresponding to a category $\Kat\subset\Lart\nlin$. Denote $\tilde G:=T_{\tau_{(N)}}GT_{\tau_{(N)}}^{-1}:=(C(G),T_{\tau_{(N)}}uT_{\tau_{(N)}}^{-1})$. Then
\begin{enumerate}	
\item $S_N\subset\tilde G\subset O_N^+$, so $\tilde G$ corresponds to some category $\tilde\Kat$,
\item $\tilde\Kat=\T_{(N)}\Kat=\{\tau_{(N)}^{\otimes l}p\tau_{(N)}^{\otimes k}\mid p\in\Kat(k,l)\}_{k,l\in\N_0}$,
\item $\T_{(N)}$ is a monoidal $*$-isomorphism.
\item If $\singleton\otimes\singleton\in\Kat$ then $G=\tilde G$.
\end{enumerate}
\end{thm}
\begin{proof}
The map $T_{\tau_{(N)}}$ is an intertwiner of $S_N$ (as any map $T_p$ for $p\in\Lart\nlin$). Thus, taking $u'$ the fundamental representation of $S_N\subset G$, we have that $T_{\tau_{(N)}}u'T_{\tau_{(N)}}^{-1}=u'$, so $S_N\subset T_{\tau_{(N)}}GT_{\tau_{(N)}}^{-1}$. In Remark \ref{R.tau} we mentioned that $T_{\tau_{(N)}}$ is orthogonal, which implies $T_{\tau_{(N)}}GT_{\tau_{(N)}}^{-1}\subset O_N^+$.

For the second point we can see that for any $p\in\Lart\nlin(k)$ we have that $u^{\otimes l}T_p=T_pu^{\otimes k}$ if and only if $v^{\otimes l}T_{\tau^{\otimes l}p\tau^{\otimes k}}=T_{\tau^{\otimes l}p\tau^{\otimes k}}v^{\otimes k}$, where $v=T_{\tau_{(N)}}uT_{\tau_{(N)}}^{-1}$.

Proving the isomorphism property is straightforward.

Finally, if $\singleton\otimes\singleton\in\Kat$, then $\tau_{(N)}\in\Kat$, which implies that $\tilde\Kat=\T_{(N)}\Kat\subset\Kat$. From the isomorphism property $\tilde\Kat=\Kat$.
\end{proof}

\begin{rem}
The implication in the last point cannot be reversed. For example, we can see that $\T_{(N)}\crosspart=\crosspart$, so $\T_{(N)}\langle\crosspart\rangle\nlin=\langle\crosspart\rangle\nlin$ although $\singleton\otimes\singleton\not\in\langle\crosspart\rangle\nlin$.
\end{rem}

\begin{ex}
\label{E.tfour}
As a non-trivial example, we can compute that
\begin{align*}
\T_{(N)}\fourpart&=\fourpart-{2\over N}(\Laaab+\Laaba+\Labaa+\Labbb)\\&+{4\over N^2}(\Laabc+\Labac+\Labbc+\Labca+\Labcb+\Labcc)-{16\over N^3}\Labcd.
\end{align*}

From Theorem~\ref{T.Tiso} it follows that the category $\langle\fourpart\rangle\nlin$ is isomorphic to the category $\langle \T_{(N)}\fourpart\rangle\nlin$ and they describe similar quantum groups. Since $\T_{(N)}\fourpart\not\in\langle\fourpart\rangle\nlin$, we see that the categories are not equal, so the quantum groups are not identical. As a consequence $\langle\T_{(N)}\fourpart\rangle\nlin$ is a non-easy linear category of partitions corresponding to a compact matrix quantum group similar to $H_N^+$ but not identical with $H_N^+$.
\end{ex}

\begin{prop}
\label{P.tauirr}
Consider a compact matrix quantum group $G$ such that $S_N\subset G\subset B_N^{\# +}$. Then $T_{\tau_{(N-1)}}G^{\rm irr}_\pm T_{\tau_{(N-1)}}^{-1}=G^{\rm irr}_\mp$, i.e.\ $G_+^{\rm irr}$ and $G_-^{\rm irr}$ are similar.
\end{prop}
\begin{proof}
It is straightforward to check that $T_{\tau_{(N-1)}}V_{(N,\pm)}=V_{(N,\mp)}$.
\end{proof}

\bibliographystyle{alpha}
\bibliography{mybase}

\end{document}